\RequirePackage{fix-cm}
\documentclass{svjour3}                     
\smartqed  
\usepackage{graphicx}
\usepackage{amsmath,amssymb,amsfonts}%
\usepackage{mathptmx}      
\usepackage{lmodern}
\usepackage{latexsym}
\usepackage{xcolor}
\usepackage{hyperref}
\usepackage{subcaption}
\usepackage{booktabs}
\usepackage{array} 
\usepackage{todonotes}

\usepackage{amsthm}
\renewcommand{\arraystretch}{1.2}
\usepackage{comment}
\usepackage{algorithm}%
\usepackage{algorithmic}%

\newcommand{\N}{\mathbb{N}}
\newcommand{\R}{\mathbb{R}}
\newcommand{\bR}{\bar{\R}}
\newcommand{\Gr}{\text{Graph }}
\newcommand{\bF}{\F^m}
\newcommand{\X}{\mathcal{X}}
\newcommand{\bX}{\bar{\X}}
\newcommand{\FuB}{-\mu,B}
\newcommand{\F}{\mathcal{F}}
\newcommand{\dom}{\text{dom}}

\newcommand{\com}{\textcolor{black}}

\newcommand{\AR}{\textcolor{black}}
\newcommand{\ARj}{\textcolor{black}}
\journalname{Mathematical Programming}
\begin{document}

\title{A constraint-based approach to function interpolation, with application to performance estimation for weakly convex optimization.\footnote{\ARj{Parts of this work, namely, parts of the introduction, a counterexample in Table~\ref{tab:counterex}, and Theorem~\ref{thm:extensive_analysis_results}, also appear in the conference tutorial paper~\cite{rubbens2023interpolation}, co-authored by the present authors. However,~\cite{rubbens2023interpolation} contains no proofs and serves solely as a high-level overview. 
}}}

\titlerunning{Constraint-based interpolation, and weakly convex performance analysis.}  

\author{Anne Rubbens \and Julien M. Hendrickx}

\institute{ICTEAM Institute, UCLouvain, Louvain-la-Neuve, Belgium\\
              \email{name.surname@uclouvain.be} 
}

\date{Received: date / Accepted: date}

\maketitle
\begin{abstract}
We consider the problem of obtaining interpolation constraints for function classes, i.e., necessary and sufficient constraints that a set of points, function values and (sub)gradients must satisfy to ensure the existence of a global function of the class considered, consistent with this set. The derivation of such constraints is crucial, e.g., in the performance analysis of optimization methods, since obtaining a priori tight performance guarantees requires using a tight description of function classes of interest. We propose an approach that allows setting aside all analytic properties of the function class to work only at an algebraic level, and to obtain counterexamples when a condition characterizing a function class cannot serve as an interpolation constraint. As an illustration, we provide interpolation constraints for the class of weakly convex functions with bounded subgradients, and rely on these constraints to outperform state-of-the-art bounds on the performance of the subgradient method on this class.
\end{abstract}
\section{Introduction}
Consider the problem of obtaining exact discrete definitions of function classes $\F$, whose need arises in any situation where one manipulates or has access to a finite set $S=\{(x_i,f_i,g_i)\}_{i=1,\ldots,N}$, and requires it to be $\mathcal{F}$\emph{-interpolable}. \AR{That is, $S$ should be consistent with some function $f \in \F$ defined everywhere, in the sense that $f(x_i)=f_i$, $g_i\in \partial f(x_i)$, $\forall x_i\in S$, where $\partial f(x_i)$ is the subdifferential of $f$ at $x_i$, see, e.g., \cite[8(3)]{rockavariational}}.

For instance, consider the problem of conducting a worst-case performance analysis of an optimization method. In this context, one derives bounds on the evolution of some performance measure, holding for any initial point satisfying some initial condition and for any element of a given function class $\mathcal{F}$, e.g., $\mathcal{F}_{0,L}$ the class of smooth convex functions, which can for example be defined as functions $f$ satisfying
     \begin{align}\label{eq:smoothconvexity}
\begin{cases}
    &f(x)\geq f(y)+\langle \nabla f(y),x-y\rangle\\& \|  \nabla f(x)- \nabla f(y)\|\leq L\|x-y\| 
\end{cases} \quad  \forall x,y \in \R^d. 
\end{align}
Obtaining bounds as tight as possible is essential to understand bottlenecks of the method analyzed as well as tune their parameters by optimizing the bound derived. 

To obtain such bounds, one typically manipulates the quantities involved in the method, that is the iterates of the method, their function values and subgradients, as well as the optimum and its function value, by combining inequalities involving these quantities and translating the knowledge that (i) they are given by the method's definition and (ii) they are evaluations of a smooth convex function,
see, e.g., \cite{nesterov2018lectures} for examples of such derivations of bounds. \com{This requires translating the global characterization of $\F$, e.g., \eqref{eq:smoothconvexity} for $\F_{0,L}$, into a discrete characterization of $\F$, i.e., conditions satisfied by any finite $\F$-interpolable dataset.}

A natural way to obtain such a characterization is to discretize global characterizations as \eqref{eq:smoothconvexity}, since any of these discretizations is necessarily satisfied at all pairs of iterates. However, while several sets of constraints can be equivalent when imposed everywhere, e.g., for $\F_{0,L}$, \eqref{eq:smoothconvexity} and      \begin{align}\label{truesmoothconvexity}
    f(x)\geq f(y)+\langle \nabla f(y),x-y\rangle+\frac{1}{2L}\|\nabla f(x)-\nabla f(y)\|^2\ \forall x,y \in \R^d \text{ \cite{nesterov2018lectures}},
\end{align}
their discretizations are not necessarily equivalent, see, e.g., \cite[Fig. 1]{taylor2017smooth} or Figure \ref{fig:nonequivalentdiscretizedconstraints}. In particular, some discretized characterizations of a class $\F$ are weaker than others, in the sense that they are satisfied by a larger region of triplets, which are not $\F$-interpolable. \com{This is the case, e.g., for \eqref{eq:smoothconvexity}. }
Any analysis relying on these weaker characterizations leads to a priori overconservative bounds, since these are valid for a larger set of admissible values than those actually corresponding to some $f\in \F$. 
\begin{figure}[!htb]
    \centering 
  \includegraphics[width=0.7\textwidth]{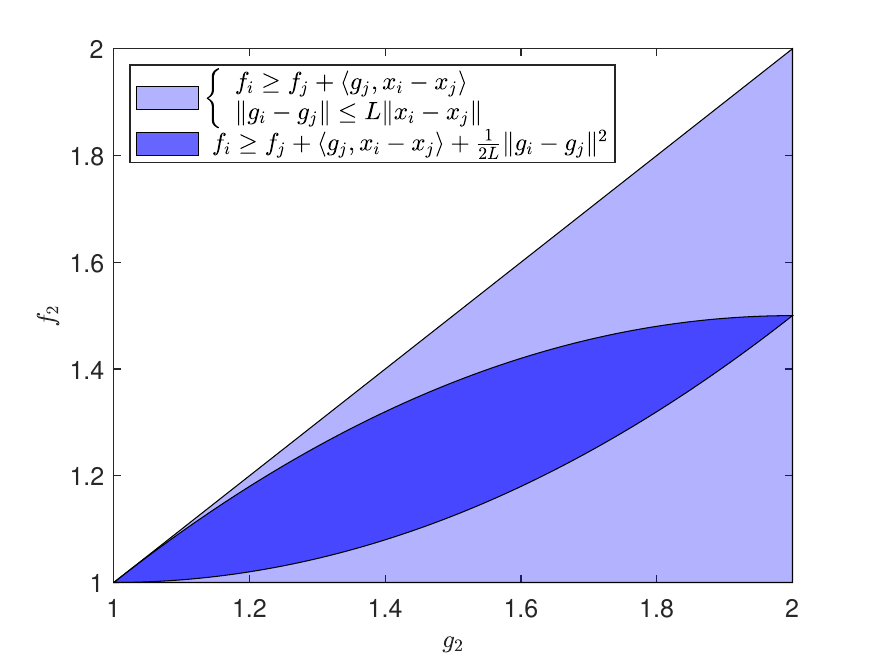}
    \caption{Given $L=1$, $(x_1,f_1,g_1)=(0,0,1)$ and $x_2=1$, allowed region for $(f_2, g_2)$ according to \eqref{eq:smoothconvexity} and \eqref{truesmoothconvexity}. The region allowed by \eqref{truesmoothconvexity} is included in and significantly smaller than the one allowed by \eqref{eq:smoothconvexity}. This translates into the fact that \eqref{eq:smoothconvexity} is weaker than the one of \eqref{truesmoothconvexity}, since satisfied by triplets that are not $\F_{0,L}$-interpolable, corresponding in this example to the difference of the two regions. \AR{From \cite[Figure 2]{rubbens2023interpolation}}} 
\label{fig:nonequivalentdiscretizedconstraints}
\end{figure}

To obtain tight bounds, one must rely on the strongest constraints possible, denoted by \emph{interpolation constraints}, \com{that is discretized definitions whose satisfaction by a dataset is not only necessary, but also sufficient to ensure its $\mathcal{F}$-interpolability.} For instance, the discretized version of \eqref{truesmoothconvexity} is an interpolation constraint for $\F_{0,L}$, see, e.g. \cite[Theorem 4]{taylor2017smooth}. \com{Hence, manipulating a set of iterates satisfying this constraint is strictly equivalent to manipulating a function $f \in \F_{0,L}$ evaluated at the iterates.} 

\com{Obtaining an exact bound on the performance of a method on given function class $\F$ does not only require describing the set of iterates via interpolation constraints, but also optimally combining these constraints and the method's definition.} While obtaining the optimal combination by hand is in general challenging, it can be computed automatically for a large number of methods, performance measures and function classes via the Performance Estimation Problem (PEP) framework \cite{drori2014performance,thesis}. This framework formulates the search of a worst-case function as an optimization problem over the set $S$ of iterates and optimum, under the constraint that $S$ is $\F$-interpolable, see Appendix \ref{app:PEP} for details. 
\subsection{Contributions} \label{sec:contributions}

Proposing interpolation constraints for a given function class, and proving their validity, are, to the best of our knowledge, open problems. In this work, we develop tools to address these challenges and illustrate their application on several examples:

\begin{itemize}
    \item \textbf{Constraint-based approach.} Rather than guessing, from a specific function class, candidate interpolation constraints for this class, we suggest a reverse approach. We start from a class of constraints and identify those that (i) define meaningful function classes $\F$ when imposed everywhere, and (ii) are interpolation constraints for $\F$ (Section~\ref{sec:interp}).
    
    \item \textbf{Algebraic assessment of interpolability.} To assert interpolability of given constraints, we introduce the notion of \emph{pointwise extensibility}, which captures whether a constraint allows consistent extension to any additional point. We prove this condition is equivalent to interpolability under reasonable regularity assumptions (Section~\ref{sec:extensibility}).
    
    \item \textbf{Systematic analysis of a class of constraints.} Starting from the class of constraints that can be written as a single inequality, linear in function values and scalar products of points and subgradients, we show that the only interpolation constraints of the class, that describe meaningful functions, are the ones describing smooth
weakly/strongly convex functions, as in \cite[Theorem 4]{taylor2017smooth}, and the one describing quadratic functions $f=\frac{\mu}{2}\|x\|^2$ (Section~\ref{sec:gram})
    \item \textbf{Automated testing of a constraint's interpolability.} Relying on pointwise extensibility, we develop a numerical heuristic to generate counterexamples to the interpolability of given constraints. We obtain such counterexamples on several classical characterizations of function classes (Section~\ref{sec:counterex}).
    \item \textbf{Tight description of a class of weakly convex functions.} Taking a step back from the constraint-based approach, but as to further illustrate pointwise extensibility, we provide interpolation constraints for \emph{weakly convex} functions, i.e., convex up to the addition of a quadratic term, with bounded subgradients. This allows improving on existing PEP-based bounds on this class \cite{das2024branch}, that relied on non-tight constraints (Section~\ref{sec:weakconv}).
\end{itemize}
\subsection{Related work}
\paragraph{Automated performance analysis.}  
Interpolation constraints applied to optimization were introduced in~\cite{taylor2017smooth} to render the PEP framework, developed in~\cite{drori2014performance}, tight. Such constraints were then obtained for a wide range of function classes and operators, including smooth (strongly/weakly) convex functions~\cite{taylor2017smooth}, smooth Lipschitz functions~\cite{thesis}, indicator functions~\cite{luner2024performance,thesis,taylor2017exact}, difference-of-convex and relatively smooth convex functions~\cite{dragomir2021optimal}, convex functions with quadratic upper bounds~\cite{goujaud2022optimal}, and monotone, cocoercive, Lipschitz, or linear operators~\cite{bousselmi2024interpolation,ryu2020operator}, see, e.g.,~\href{https://pepit.readthedocs.io/en/latest/api/functions_and_operators.html}{PEPit documentation}~\cite{goujaud2024pepit},\cite{rubbens2023interpolation} for surveys.

These constraints were combined with the PEP framework to obtain tight performance guarantees for various first-order methods~\cite{colla2023automatic,de2017worst,drori2020complexity,dragomir2021optimal,gannot2022frequency,guille2022gradient,goujaud2024pepit,kamri2023worst,lee2025convergence,taylor2019stochastic,taylor2018exact}, and to design new optimal methods~\cite{barre2022principled,drori2017exact,drori2020efficient,das2024branch,kim2016optimized,kim2018adaptive}. Interpolation constraints are also used in the Integral Quadratic Constraints (IQC) framework, an automated approach relying on tools from control theory to certify iteration-wise progress~\cite{lessard2016analysis,taylor2018lyapunov,van2017fastest}. In contrast to the approach proposed here, based on pointwise extensibility and constraint structure, the techniques used in this context typically rely on analytical properties of the function classes of interest.

\paragraph{Extension of continuous functions.}  
Interpolation constraints can be also be viewed as conditions ensuring the extension of a property outside of its domain, a topic extensively studied since Whitney's foundational work on continuous extensions~\cite{whitney1934differentiable,whitney1992analytic}; see also~\cite{brudnyi1994generalizations,brudnyi2001whitney,brudnyi2011methods,fefferman2009whitney} and references therein. Most results in this literature differ from the approach we propose in several ways. First, some focus on qualitative extensions where regularity constants may increase, see, e.g.,~\cite{azagra2017whitney}, while we seek quantitative guarantees, essential in the context of establishing properties related to these constants. Second, some results require assumptions on the domain, such as its convexity~\cite{azagra2019smooth,yan2012extension}, whereas we only impose conditions on finite sets of iterates. Third, certain works address interpolation of function values only~\cite{brudnyi1994generalizations,fefferman2009whitney,fefferman2017interpolation}, while our goal includes subgradient interpolation.

Still, some settings align closely with our framework, such as the extension of Hölder-Lipschitz functions~\cite{mcshane1934extension,minty1970extension} and that of smooth convex functions~\cite{azagra2017extension,gruyer2009minimal}. These last two works treat the extension problem algebraically, via \emph{infinite} pointwise extensibility, requiring that any new point can be added to an arbitrary (possibly infinite) dataset satisfying a given constraint $p$. However, such infinite extensions are typically harder to verify than the \emph{finite} pointwise extensibility we focus on, motivating the perspective adopted here.\vspace{-0.3cm}

\subsection{Notation}
Throughout, we consider a Euclidean space $\mathbb{R}^d$ with standard inner product $\langle \cdot,\cdot\rangle$ and induced norm $\|\cdot\|$. We denote by $\mathbb{Q}^d$ the set of rationals in $\R^d$, by $x^{(k)}$ the $k^{\text{th}}$ component of any $x\in \R^d$ and by $\mathcal{B}_{c,R}:=\{y \in \mathbb{R}^d:\ \|y-c\|\leq R\}$ a sphere of center $c \in \R^d$ and radius $R \in \mathbb{N}$. Given $N \in \mathbb{N}$, we let $[N]=\{0, \ldots,N\}$. 

Given $f:\R^d\to \R$, we say $f$ is Lipschitz continuous if \cite{nesterov2018lectures}  $ \|f(x)-f(y)\|\leq L\|x-y\| \ \forall x,y\in \R^d$, $f$ is smooth if differentiable everywhere with first derivative Lipschitz continuous, $f$ is convex if $    f(\alpha x+(1-\alpha)y)\leq \alpha f(x)+(1-\alpha)f(y) \ \forall x, y \in \R^d, \ \alpha \in (0,1),$  and for any $\rho \geq 0$ $f$ is $\rho$-weakly (respectively $\rho$-strongly) convex if $f+\rho \frac{\|\cdot\|^2}{2} $ (respectively $f-\rho \frac{\|\cdot\|^2}{2}$) is convex. 

We refer by $\partial f(x)$ to the regular subdifferential  \cite[8(3)]{rockavariational} of $f$ at $x$, defined as the set of vectors $v\in \R^d$, satisfying, $\forall y\in\R^d$,  \begin{align}\label{eq:subdiff_general}
    f(y)\geq f(x)+\langle v,y-x\rangle +\omega(\|x-y\|).
\end{align}
For $f$ smooth, $\partial f(x)=\{\nabla f(x)\}$, and for $f$ (strongly) convex, it reduces to the classical subdifferential. By an abuse of notation, we denote any $v\in\partial f(x)$ a subgradient of $f$ at $x$. 

\ARj{Finally, we use the notation $\rightrightarrows$ to denote a multi-valued mapping; that is, given $\X\subseteq\R^d$, $F \colon \X \rightrightarrows \R^n$ denotes a mapping that may associate multiple values in $\R^n$ to each point $x \in \X$. We refer to any of these values by $F_x\in F(x)$. We denote by $\Gr F:=\{(x,F_x)\in \X\times \R^n \big| \ F_x\in F(x)\}$ the graph of $F$.}
\section{Constraint-based approach to interpolation}\label{sec:interp}
\AR{This paper considers function classes $\F$ characterized via \emph{pairwise} algebraic constraints $p$, that involve evaluations of a function $f$ and its (sub)gradient, see, e.g., \eqref{eq:smoothconvexity} and \eqref{truesmoothconvexity}.} \AR{Formally, given any $\X \subseteq \R^d$, let $$F:\X\rightrightarrows \R\times\R^d:x \rightrightarrows (f(x),g(x)),$$ denote a set-valued mapping, where $f(x)$ is single-valued and $g(x)$, possibly multi-valued, is intended to represent a subgradient of $f$ at $x$, but is an a priori unrelated quantity. We consider constraints
\begin{align*}
    p:(\R^d \times \R \times \R^d)^2 \to \R^\ell, \quad \ell \in \mathbb{N},
\end{align*}
and say that $p$ is satisfied at a pair $(x_1, x_2) \in \X$ if $    p((x_1, F_{x_1}), (x_2, F_{x_2})) \geq 0 \quad \text{for all } F_{x_i} \in F(x_i), \ i=1,2,$ where the inequality is to be taken elementwise for $\ell \geq 2$.}

The proposed constraint-based approach stems from the observation that any pairwise constraint defines a function class over $\X \subseteq \R^d$, simply by requiring the constraint to hold at all pairs of points in $\X$.

\begin{definition}\label{def:class}
Given a pairwise constraint $p$ and $\X \subseteq \R^d$, a function class defined by $p$ on $\X$ consists of all mappings $F$ satisfying $p$ on $\X$.
\AR{
\begin{align*}
    \F_p(\X) := \big\{ F : \X \rightrightarrows \R \times \R^d \ \big| \ 
    &p((x_1, F_{x_1}), (x_2, F_{x_2})) \geq 0, \\
    &\forall x_1,x_2 \in \X, \ \forall F_{x_i} \in F(x_i) \big\}.
\end{align*}}
\end{definition}

\AR{Further, we say $F\in \F_p(\R^d)$ is \emph{maximal} if there exists no proper extension of $F$ in $\F_p(\R^d)$. We denote by $\bF_p(\R^d)$ the class of maximal functions in $\F_p(\R^d)$. This class typically encompasses functions satisfying $p$ \emph{for all subgradients} in their subdifferential, rather than possibly for a single one.
\begin{definition}
    Let $p$ a pairwise constraint. We let 
    \begin{align}
        \bF_p(\R^d):=\{F:\R^d\rightrightarrows\R\times\R^d\big|\ &F\in \F_p(\R^d), \\& \not\exists \bar F\in\F_p(\R^d) \text{ s.t. } \Gr(F)\subset \Gr(\bar F)\}.\nonumber
    \end{align}
\end{definition}}
\AR{Our goal, starting from classes of pairwise constraints, is to only keep those that (i) imply a subdifferential relationship between $f$ and $g$ when imposed everywhere, and (ii) are interpolation constraints. We thus formally present these two properties.}

\paragraph{Differentially consistent constraints} Notice that at this stage, given $F:\X \rightrightarrows \R \times \R^d:x \rightrightarrows (f(x), g(x))$, with $F \in \F_p(\X)$, $f(x)$ and $g(x)$ are related only via the algebraic relation imposed by $p$. We are interested in \emph{differentially consistent} constraints, i.e., constraints which, when satisfied everywhere, imply $g(x) \in \partial f(x)$. 
\AR{
\begin{definition}[Differentially consistent constraint]
A pairwise constraint $p$ is \emph{differentially consistent} if, given $\{(x_i,f_i,g_i),(x_j,f_j,g_j)\}\in (\R^d\times \R\times \R^d)^2$, the condition
\[
p((x_i,f_i,g_i),(x_j,f_j,g_j)) \geq 0,
\]
implies
\begin{equation}\label{eq:subdiff_gx}
    f_j \geq f_i + \langle g_i, x_j - x_i \rangle + \omega(\|x_i - x_j\|).
\end{equation}
\end{definition}}

\begin{example}[Convex functions] The pairwise constraint $p$ characterizing convex functions on their effective domain, and defined as \begin{align}\label{eq:conv}
    &p((x_i,f_i,g_i),(x_j,f_j,g_j))= f_j-f_i-\langle g_i,x_j-x_i\rangle, 
    \end{align}
    is differentially consistent since its satisfaction implies $f_j\geq f_i+\langle g_i,x_j-x_i\rangle$, hence \eqref{eq:subdiff_gx}. 
    \label{exconvsmooth}
\end{example}
\begin{example}\label{trivialp}
    The trivial pairwise constraint $p((x_i,f_i,g_i),(x_j,f_j,g_j))=0$ is not differentially consistent since satisfied by all mappings $F$ on $\R^d$, e.g., by $f(x)=0, \ g(x)=x, \forall x \in \R^d$, which do not satisfy $g(x)\in \partial f(x)$.
\end{example}

\paragraph{Interpolation constraints} 
Trivially, for any constraint $p$ and any $\X_1 \subseteq \X_2$, a mapping $F_2 \in \F_p(\X_2)$ can always be restricted to a mapping $F_1 \in \F_p(\X_1)$. However, the converse does not hold in general. In particular, a mapping satisfying $p$ on some $\X \subset \R^d$ is not necessarily extendable to a mapping in $\bF_p(\R^d)$, see, e.g., Figure~\ref{fig:nonequivalentdiscretizedconstraints} for an example. We are interested in constraints $p$ for which such an extension is always possible; we refer to these as \emph{interpolation constraints}.

\begin{definition}[Interpolation constraint] \label{def:intrp}
We say a pairwise constraint $p$ is an \emph{interpolation constraint} for $\bF_p(\R^d)$, if
\[
\forall \text{ finite } \X \subset \R^d, \ \forall F \in \F_p(\X),\quad \exists \bar F \in \bF_p(\R^d) \text{ such that }  F(x) = \bar F(x) \ \forall x \in \X.
\]
\end{definition}

The approach to interpolation of Definition \ref{def:intrp} differs from what is classically done in, e.g., \cite{taylor2017smooth}. Instead of defining $p$ from a function class $\F$, as a necessary and sufficient condition to impose on any finite set $\X$ to ensure its $\F$-interpolability, we directly focus on $p$ as an interpolation constraint for the class it implicitly defines. \com{Nevertheless, these are for almost all practical purposes equivalent notions, see Appendix \ref{app:equiv} for details.} 
\begin{remark}
\ARj{Throughout, to avoid additional technical difficulties, we restrict attention to finite-valued functions. However, one could, in different ways, explicitly allow function classes whose elements are defined on \( \mathbb{R}^d \), but take finite values only on an \emph{effective domain} \( \bar{\mathcal{X}} \subseteq \mathbb{R}^d \), assigning infinite values (and correspondingly empty subdifferentials) outside \( \bar{\mathcal{X}} \). For instance, one could introduce the function class
\begin{align}\label{eq:infinnte}
    \F^{m,\infty}_p(\R^d):=\bigg\{F:\R^d\rightrightarrows \bR\times (\R^d\cup\emptyset)\big |&\ \exists \bX \subseteq \R^d\text{ convex s.t. }  F_{|\bX} \in \bF_p(\bX)\nonumber \text{ and}\\
     F(x)&=(+\infty,\emptyset), x\in \R^d\setminus \bX  \bigg\},\end{align}
     where $F_{|\bX}$ denotes the restriction of $F$ to $\bX$, and require extension to $\F^{m,\infty}_p(\R^d)$ instead of $\bF_p(\R^d)$.}

\end{remark}
\paragraph{Compact notation for a pairwise constraint}
\AR{Let $p$ be a pairwise constraint, and consider some finite set $\X \subset \R^d$, along with a mapping $F \in \F_p(\X)$. We denote by $S = \{(x_i, f_i, g_i)\}_{i \in [N]} \subset (\R^d \times \R \times \R^d)^N$ the combination of all points in $\X$ and their associated values under $F$. That is, if $F$ is multi-valued, $S$ may include multiple entries with the same $x_i$, associated to different values of $g_i$. We denote by $p^{ij} := p\left((x_i, f_i, g_i), (x_j, f_j, g_j)\right)$ an evaluation of $p$, and we say that $S$ satisfies $p$ if $p^{ij} \geq 0$ for all $i,j \in [N]$.}

\AR{Finally, when clear from the context, and given $F = (f, g) \in \bF_p(\R^d)$, we denote by $p^{xy} := p\left((x, f(x), g_x), (y, f(y), g_y)\right)$ any evaluation of $p$ at points $x, y \in \R^d$, for some $(f(x), g_x) \in F(x)$ and $(f(y), g_y) \in F(y)$.}

\section{Pointwise extensible constraints} \label{sec:extensibility}
Definition \ref{def:intrp} allows for an algebraic approach to establish interpolability, based on idea of extending \com{a mapping satisfying a constraint $p$ on a finite set to just one (arbitrary) point, in the spirit of what is done in, e.g., \cite{azagra2017extension,gruyer2009minimal} with respect to potentially infinite sets.}
 \begin{definition}[Pointwise extensible constraint] \label{def:pointwise_ext}
We say a pairwise constraint $p$ is pointwise extensible if $$\forall \text{ finite }\X\subset\R^d, \forall z\in\R^d, \ \forall F \in \F_p(\X),\ \exists \bar F\in \F_p(\X\ \cup \ z )\text{ s.t. }F(x)=\bar F(x)\  \forall x\in \X.$$
\end{definition}
 \begin{example}[Convex functions]\label{ex:convex_pointwise}
Let $p$ be defined as in \eqref{eq:conv}. For some $N \in \N$, consider $S=\{(x_i,f_i,g_i)\}_{i\in [N]}$, where $S$ satisfies $p$. Given any $x\in \R^d$, let 
  \begin{align*}
    i_\star=\underset{i\in [N]}{\text{argmax }} f_i+\langle g_i,x-x_i\rangle, \quad  f= f_{i_\star}+\langle g_{i_\star},x-x_{i_\star}\rangle \text{ and } g=g_{i_\star}. 
\end{align*}
By definition of $i_\star$, $f\geq f_i+\langle f_i,x-x_i\rangle \ \forall i\in [N].$
In addition, since $p^{ji_\star}\geq0$, 
  \begin{align*}
    f_j\geq f_{i_\star}+\langle g_{i_\star},x_j-x_{i_\star}\rangle =f+\langle g,x_j-x\rangle \
     \forall j\in [N]. 
\end{align*}
Hence, these $f$ and $g$ extend the constraint to $x$ and $p$ is pointwise extensible.
\end{example}

Under some regularity assumptions, we show equivalence between pointwise extensible and interpolation constraints. By definition, the first implication holds.
 \begin{proposition}
    Let $p$ be a pairwise constraint. If $p$ is an interpolation constraint for $\bF_p(\R^d)$, then $p$ is pointwise extensible.
    \label{interptoext}
\end{proposition}

Proposition \ref{interptoext} allows deriving counterexamples to the interpolability of a constraint $p$, see Section \ref{sec:counterex} for an illustration. \AR{The converse holds under additional technical conditions. Indeed, pointwise extensibility allows iteratively extending any $F \in\F_p(\X)$ to the rationals, i.e., to $\bar F\in \F_p(\mathbb{Q}^d)$. Then, one can extend $\bar F$ to $\R^d$, taking its limit on $\mathbb{Q}^d$, provided (i) such a limit exists, and (ii) $p$ is continuous, hence still satisfied at the limit. We formalize these requirements by introducing the notion of a regularly extensible constraint, before providing sufficient conditions ensuring regular extensibility of a constraint $p$.}
\begin{definition}[\AR{Regularly extensible constraint}]
~\\
Let $p$ be a pairwise constraint. We say $p$ is \emph{regularly extensible} if 
\begin{enumerate}
    \item $p$ is continuous.
    \item Given a finite set $\X\subset\R^d$, and $F\in\F_p(\X)$, $$\exists \bar F\in \F_p(\mathbb{Q}^d\cup \X): \ F(x)=\bar F(x), \ x\in \X,$$ such that alongside any Cauchy sequence $(x_k)_{k=1,2,\ldots}\subset\mathbb{Q}^d$, there exists a convergent subsequence in $(\bar F(x_k))_{k=1,2,\ldots}$.
\end{enumerate} \label{def:regext}
\end{definition}
\AR{\begin{lemma}\label{lem:reg_ext_example}
    Let $p$ be a continuous, differentially consistent, pointwise extensible, pairwise constraint. If any of the following conditions holds:
\begin{itemize}
    \item Given $\{(x_i,f_i,g_i),(x_i,f_i,g_i)\}\in (\R^d\times \R\times \R^d)^2$, satisfaction of $p^{ij}$ and $p^{ji}$ implies
    \begin{align}\label{eq:reg}
        \|g_i-g_j\|\leq \omega(\|x_i-x_j\|);
    \end{align}
    \item Given $\{(x_i,f_i,g_i),(x_i,f_i,g_i)\}\in (\R^d\times \R\times \R^d)^2$, satisfaction of $p^{ij}$ and $p^{ji}$ implies
    \begin{align}\label{eq:reg1}
        \|g_i\|\leq A\|g_j\|+B, \text{ where $A,B<+\infty$, or;}
    \end{align}
    \item Given any finite $\X \subset\R^d$ and $z\in \R^d$, $ \forall F=(f,g)\in \F_p(\X),$
    \begin{align}\label{eq:reg3}
       &\exists \bar F\in \F_p(\X\ \cup \ z )\nonumber\\\text{ s.t. } &\begin{cases}
            F(x)=\bar F(x)\  \forall x\in \X\\
             \bar F(z)=(\bar{f}(z),\bar{g}(z)) \text{ satisfies } \|\bar g(z)\|\leq \underset{x_i \in \X, \ g_i\in g(x_i)}{\max}\|g_i\|;
        \end{cases} 
    \end{align}
\end{itemize}
then $p$ is regularly extensible.
\end{lemma}}
\begin{proof}
    Given $F\in\F_p(\X)$, we construct an extension $\bar F=(\bar f,\bar g)\in \F_p(\mathbb{Q}^d\cup\X)$, such that alongside any Cauchy sequence $(x_k)_{k=1,2,\ldots}\subset\mathbb{Q}^d$, (i) $(\bar g(x_k))_{k=1,2,\ldots}$ is bounded, hence contains a convergent subsequence \cite[3.4.8]{bartle2000introduction}, and (ii) $(\bar f(x_k))_{k=1,2,\ldots}$ is a Cauchy sequence, hence a convergent sequence \cite[3.5.5]{bartle2000introduction}.
    
    Suppose first \eqref{eq:reg} or \eqref{eq:reg1} is satisfied. By countability of $\mathbb{Q}^d$ and pointwise extensibility of $p$, $F$ can be iteratively extended to $\mathbb{Q}^d\cup\X$ in a way as to satisfy $p$. Denote by $\bar F\in \F_p(\mathbb{Q}^d\cup\X)$ this extension of $F$, and consider any Cauchy sequence $(x_k)_{k=1,2,\ldots}\subset\mathbb{Q}^d$. Suppose \eqref{eq:reg} is satisfied. Alongside $(x_k)_{k=1,2,\ldots}$, since Lipschitz continuity preserves Cauchy sequences, see, e.g., \cite[Theorem 5.4.7]{bartle2000introduction}, $\bar g(x_k)$ is a Cauchy sequence, hence bounded \cite[3.4.8]{bartle2000introduction}. On the other hand, if \eqref{eq:reg1} is satisfied, then
    \begin{align*}
        \|\bar g(x_k)\|\leq \min_{x_i\in \X, \ g_i\in g(x_i)}A\|g_j\|+B<+\infty, \ \forall k=1,2,\ldots,
    \end{align*}
    hence $(\bar g(x_k))_{k=1,2,\ldots}$ is bounded. Suppose now $p$ satisfies \eqref{eq:reg3}. By countability of $\mathbb{Q}^d$ and pointwise extensibility of $p$, $F$ can be iteratively extended to $\mathbb{Q}^d\cup\X$ in a way as to satisfy $p$, and $\bar g(x)\leq \max_{x_i\in \X, \ g_i\in g(x_i)}\|g_i\|$, at all $x\in \mathbb{Q}^d$. Denote by $\bar F\in \F_p(\mathbb{Q}^d\cup\X)$ this extension of $F$, and consider any Cauchy sequence $(x_k)_{k=1,2,\ldots}\subset\mathbb{Q}^d$. Alongside this sequence, \begin{align*}
        \|\bar g(x_k)\|\leq \max_{x_i\in \X, \ g_i\in g(x_i)}\|g_i\| \ \forall k=1,2,\ldots,
    \end{align*}
    hence $(\bar g(x_k))_{k=1,2,\ldots}$ is bounded. In all cases, denote by $M<+\infty$ the bound on $(\bar g(x_k))_{k=1,2,\ldots}$.

    In addition, since $p$ is differentially consistent, satisfaction of $p$ at any pair $(y_i,y_j) \in \R^d$ implies satisfaction of \eqref{eq:subdiff_gx} at $(y_i,y_j)$ and $(y_j,y_i)$, i.e., 
\begin{align*}
     \langle\bar g(y_j),y_i-y_j\rangle+\omega(\|y_i-y_j\|)\leq \bar f(y_i)-\bar f(y_j)&\leq \langle\bar g(y_i), y_i-y_j\rangle -\omega(\|y_i-y_j\|)\\
     -\| \bar g(y_j)\| \|y_i-y_j\|+\omega(\|y_i-y_j\|)\leq \bar f(y_i)-\bar f(y_j)&\leq \|\bar g(y_i)\|\|y_i-y_j\| -\omega(\|y_i-y_j\|)\\
    \Leftrightarrow  |\bar f(y_i)-\bar f(y_j)|&\leq M\|y_i-y_j\| +|\omega(\|y_i-y_j\|)|.
\end{align*}
By definition of $\omega(\cdot)$, $\forall \varepsilon>0$, $\exists \delta_1\geq 0: |t|\leq \delta_1 \Rightarrow |\omega(|t|)|\leq \varepsilon |t|$. Hence, for any $\varepsilon>0$, let $\delta=\min\{\frac{\varepsilon}{M+\varepsilon},\delta_1\}$. Then,
\begin{align*}
    \text{If }\|y_i-y_j\|\leq \delta, \text{ then } |\bar f(y_i)-\bar f(y_j)|\underset{\delta \leq \delta_1}{\leq} (M+\varepsilon)\delta \underset{\delta \leq \frac{\varepsilon}{M+\varepsilon}}{\leq } \varepsilon. 
\end{align*}
Hence, since $(y_j)_{j=1,2,\cdots}$ is a Cauchy sequence, so is $(\bar f(y_j))_{i=1,2,\cdots}$.
\end{proof}
For instance, by Example \ref{ex:convex_pointwise}, the characterization of convex functions given in \eqref{eq:conv} is regularly extensible. In addition, any continuous differentially consistent pointwise extensible constraint ensuring some continuity on subgradients, e.g., Lipschitz continuity or boundedness, is regularly extensible. We now show that any regular pointwise extensible constraint is an interpolation constraint.
\begin{theorem}\label{theoequi}
Let $p$ be a pairwise constraint. \AR{If $p$ is a regular pointwise extensible constraint, then $p$ is an interpolation constraint for $\bF_p(\R^d)$.}
\end{theorem}
\begin{proof}
Let $N \in \N$, $\X=\{x_i\}_{i\in [N]}$ $\subset \R^d$ and $F=(f,g)\in \F_p(\X)$. We construct $\bar F \in \F_p(\R^d)$ satisfying $F(x_i)=\bar F(x_i), \ \forall i \in [N]$. 

For every $x \in \X$, we set $\bar F(x)=F(x)$. Then, we extend $\bar F$ to the rationals as follows. By regular extensibility of $p$, there exists a mapping $\bar F=(\bar f,\bar g) \in \F_p(\mathbb{Q}^d \cup \X)$, satisfying $F(x_i)=\bar F(x_i), \ \forall i \in [N]$, and containing convergent subsequences along any Cauchy sequence in $\mathbb{Q}^d$. For each remaining $x \in \R^d \setminus\mathbb{Q}^d$, we first define a value for $\bar F(x)$ and then establish that this value still satisfies $p$. Since $\mathbb{Q}^d$ is dense in $\mathbb{R}^d$, there exists a sequence of rationals $(y_j)_{j=1,2,\ldots}$ converging to $x$. Alongside this sequence, consider $\bar F(y_j)=(\bar f(y_j),\bar g(y_j))$. By regular extensibility of $p$, there exists subsequences of $\bar f(y_j),\ \bar g(y_j)$ converging to some $\mathbf{\bar f}$, $\mathbf{\bar g}$, respectively. Set $ \bar F(x)=(\mathbf{\bar f},\mathbf{\bar g})$, i.e. to the limits of $\bar F(y_j)$ alongside the sequence $y_j$ converging to $x$. By continuity of $p$, $\bar F(x)$ extends $p$ to $x$. Hence, $F$ can be extended to some $\bar F(x)\in\F_p(\R^d)$. It remains to show that $\bar F(x)$ admits a \emph{maximal} extension to $\bF_p(\R^d)$.

Let $\mathcal{P}=\{H \in \F_p(\R^d)\big| \ \Gr\bar F\subseteq\Gr H\},
   $ which is non-empty, along with the partial order on $\mathcal{P}$  \vspace{-2.7mm} \begin{align*}(H_1\preceq H_2) \Leftrightarrow \Gr H_1 \subset \Gr H_2 , \ \forall H_1,H_2\in \mathcal{P}\end{align*}
   Given any chain $\mathcal{C}$ in $\mathcal{P}$, let $\bar H$ be defined as \begin{align*}
       \Gr \bar H=\cup_{H\in \mathcal{C}}\ \Gr H.
   \end{align*}
   Then $\bar H \in \mathcal{P}$, and $\bar H$ is an upper bound for $\mathcal{C}$. Hence, Zorn's lemma guarantees the existence of a maximal element $h \in \mathcal P$, which extends $\bar F$ by definition of $\mathcal{P}$, and cannot be extended since maximal in $\mathcal{P}$.
\end{proof}
\AR{\begin{remark}
    While we restrict ourselves here to function classes defined by pairwise constraints involving evaluations of a single function and its subgradient, all concepts presented in this section can be straightforwardly extended to broader settings. These include, e.g., constraints involving more than two points, constraints involving sums or differences of functions, constraints involving higher-order derivatives, or constraints characterizing operator classes.
\end{remark}
}

We now demonstrate how pointwise extensibility can be used to automatically compute counterexamples to the interpolability of given constraints, and to verify the interpolability of specific ones. We then illustrate the constraint-based approach introduced in Section~\ref{sec:interp} by exhaustively analyzing a class of constraints.

\section{A numerical program to generate counterexamples}\label{sec:counterex} 
By Proposition~\ref{interptoext}, to obtain a counterexample to the interpolability of a constraint $p$, and thereby discard it as a discrete characterization of the class it defines, it suffices to exhibit a dataset $S = \{(x_i, f_i, g_i)\}_{i \in [N]}$ satisfying $p$, together with a point $x \in \R^d$ at which $p$ cannot be extended. Given $S$ and $x$, verifying whether such an extension is possible reduces to solving an optimization problem.

\begin{lemma}
    Given $N\in \N$, consider a pairwise constraint $p$, a set $S = \{(x_i, f_i, g_i)\}_{i \in [N]}$ satisfying $p$, and $x \in \R^d$.  Let
    \begin{align}
        \label{eq:innerprob}
        \tau_\star:=\min_{\substack{\tau \in \R,\\ f_x \in \R,\ g_x \in \R^d}} \tau \quad \text{s.t.} \quad p^{xi} \geq -\tau \text{ and } p^{ix} \geq -\tau, \quad \forall i \in [N].
    \end{align}
    Then, $S$ admits an extension satisfying $p$ at $x$ if and only if $\tau_\star \leq 0$.
    \label{lem:exti}
\end{lemma}

\begin{proof}
    The constraint $p$ can be extended at $x$ if there exists a pair $(f_x, g_x)$ such that the extended dataset $\{(x_i, f_i, g_i)\}_{i \in [N]} \cup \{(x, f_x, g_x)\}$ satisfies $p$, i.e., if
$\exists (f_x, g_x) \text{ s.t. }$ $ p^{xi} \geq 0 \text{ and } p^{ix} \geq 0, \quad \forall i \in [N].$ Problem~\eqref{eq:innerprob} computes the smallest relaxation $\tau$ needed to satisfy all these inequalities. If $\tau_\star \leq 0$, an extension exists.
\end{proof}

Whenever $\tau_\star$, the solution to~\eqref{eq:innerprob}, is positive, $S$ and $x$ form a counterexample to the interpolability of $p$. Thus, for a given $N \in \N$, one can search for such a counterexample by identifying the "hardest" $x$ and $S$, i.e., those maximizing $\tau_\star$.

\begin{proposition}
    Given $N\in \N$, consider a pairwise constraint $p$, and let
  \begin{equation}\label{prob:counterex}
     \begin{aligned}
   & \textcolor{white}{blablablabla}\tau_\star(N):= && \max_{\substack{x \in \R^d\\ S=\{(x_i,f_i,g_i)\}_{i\in [N]}}}\min_{\substack{\tau\in \R,\\ f_x \in \R,\ g_x\in \R^d}} \tau \\
     &\text{Extension of $p$ at $x$}&&\text{ s.t. } p^{xi}\geq -\tau\text{ and } p^{ix}\geq -\tau, \ \forall i \in [N],  \\
     &\text{Satisfaction of $p$ by $S$}&&\quad \quad p^{ij}\geq 0, \ \forall i ,j\in [N].
\end{aligned}
\end{equation}
    Then, $p$ is a pointwise extensible constraint if and only if $\tau_\star(N) \leq 0$, $\forall N \in \N$.
\end{proposition}

\begin{proof}
Let $\tau_\star(N)>0$ for some $N\in \N$. By Lemma \ref{lem:exti}, there exists a counterexample to the pointwise extensibility of $p$. The second direction holds by Lemma \ref{lem:exti}.
\end{proof}

\AR{In general, \eqref{prob:counterex} is not directly solvable due to its non-convexity and the combinatorial nature of the search over datasets. To address this, we use a heuristic that iteratively constructs a dataset maximizing the solution to~\eqref{eq:innerprob}; see the code referenced in Section~\ref{sec:conclu} for implementation details.} While this heuristic does not provide guarantees of interpolability in the absence of a counterexample, it effectively obtained such counterexamples for several classically used constraints, as illustrated in Table~\ref{tab:counterex}.

\begin{table}[ht!]
\renewcommand{\arraystretch}{1.2}
\centering
\begin{tabular}{@{}p{2.7cm} p{5.8cm} p{6.5cm}@{}}
\toprule
\textbf{Function class} & \textbf{Classical characterization} & \textbf{Counterexample} \\
\midrule
$\begin{array}{l}
\textbf{Weakly convex}\\\textbf{
functions}\end{array}$ & 
$\begin{array}{l}
\text{For some } B,\mu \geq 0,\ \forall x,y \in \mathbb{R}^d, \\
\forall g_y \in \partial f(y)  \cite{davis2019stochastic,das2024branch}: \\
\begin{aligned}
\begin{cases}
    &   f(x) \geq f(y) + \langle g_y, x - y \rangle - \frac{\mu}{2} \|x - y\|^2, \\
&\|g_y\| \leq B
\end{cases}
\end{aligned}
\end{array}$ & 
$\begin{array}{l}
S = \left\{ 
\begin{pmatrix} 0\\0\\B \end{pmatrix},
\begin{pmatrix} \frac{2B}{\mu}+1\\ \frac{\mu}{2}-2B \\ B \end{pmatrix}
\right\}, \\
x = \frac{2B}{\mu}
\end{array}$ \\

\midrule
$\begin{array}{l}
\textbf{Uniformly convex}\\\textbf{
functions}\end{array}$  &
$\begin{array}{l}
\text{For some } p > 2,\ \forall x, y \in \mathbb{R}^d \cite{doikov2021minimizing}: \\
f(x) \geq f(y) + \langle \nabla f(y), x - y \rangle + \frac{\mu}{p} \|x - y\|^p
\end{array}$&
\AR{$\begin{array}{l}
S = \left\{
\begin{pmatrix} 0\\0\\0 \end{pmatrix},
\begin{pmatrix} 2\\ \frac{\mu 2^p}{p}\\ \frac{\mu 2^p}{p} \end{pmatrix},
\begin{pmatrix} 1\\ \frac{\mu(2^{p-1} - 1)}{p}\\ \frac{\mu 2^{p-1}}{p} \end{pmatrix}
\right\}, \\
x = 0.1
\end{array}$} \\
\midrule
$\begin{array}{l}
\textbf{Holder smooth}\\\textbf{
functions}\end{array}$  &
$\begin{array}{l}
\text{For some } \alpha \in (0,1),\ \forall x, y \in \mathbb{R}^d \cite{doikov2021minimizing}: \\
f(x) \leq f(y) + \langle \nabla f(y), x - y \rangle + \frac{L}{\alpha+1} \|x - y\|^p
\end{array}$ &
\AR{$\begin{array}{l}
S = \left\{
\begin{pmatrix} 0\\0\\0 \end{pmatrix},
\begin{pmatrix} 2\\  L(1+\frac{1}{\alpha})^\alpha\\ L(1+\frac{1}{\alpha})^\alpha\end{pmatrix},
\begin{pmatrix} 1\\ \frac L2 (1+\frac{1}{\alpha})^\alpha-\frac{L}{\alpha+1}\\\frac L 2(1+\frac{1}{\alpha})^\alpha\end{pmatrix}
\right\}, \\
x = 0.5
\end{array}$} \\

\midrule
$\begin{array}{l}
\textbf{Smooth functions}\\\textbf{
with P\L{} condition}\end{array}$ &
$\begin{array}{l}
\text{For some } 0 \leq \mu \leq L,\ \forall x, y \in \mathbb{R}^d \cite{karimi2016linear}: \\
\begin{aligned}
    \begin{cases}
        &f(x) \leq f(y) + \langle \nabla f(y), x - y \rangle + \frac{L}{2} \|x - y\|^2, \\
&\|g(x)\|^2 \geq 2\mu(f(x) - f_\star),\\& f(x) \geq f_\star,\quad g_\star = 0
    \end{cases}
\end{aligned}
\end{array}$ &
$\begin{array}{l}
S = \left\{
\begin{pmatrix} 0\\0\\0 \end{pmatrix},
\begin{pmatrix} 1\\ \frac{L}{2} \\ L \end{pmatrix}
\right\}, \mu = \frac{L}{2}, \\
x = 0.5
\end{array}$ \\

\midrule
$\begin{array}{l}
\textbf{Functions with}\\\textbf{
Lipschitz Hessian}\end{array}$ &
$\begin{array}{l}
\text{For some } M \geq 0,\ \forall x, y \in \mathbb{R}^d \cite{nesterov2018lectures}: \\
|f(x) - f(y) - \langle \nabla f(y), x - y \rangle \\
\quad - \tfrac{1}{2}(x - y)^T \nabla^2 f(y)(x - y)| \leq \tfrac{M}{6} \|x - y\|^3
\end{array}$ &
$\begin{array}{l}
S = \left\{
\begin{pmatrix} 0\\0\\0\\0 \end{pmatrix},
\begin{pmatrix} 1\\ -\frac{M}{6}\\ 0\\ 0 \end{pmatrix}
\right\}, \\
x = 0.5
\end{array}$ \\

\bottomrule
\end{tabular}
\caption{Function classes whose classical constraints are not pointwise extensible, and corresponding counterexamples.}
\label{tab:counterex}
\end{table}
  \begin{remark}
    The bounds on the performance of optimization methods on the function classes presented in Table \ref{tab:counterex}, derived based on their discretized characterizations, are a priori not tight, since they apply to datasets that are not necessarily consistent with functions of these classes. In particular, there exist questions for which one cannot obtain a tight answer based on these discretizations.
\end{remark} 

To the best of our knowledge, no interpolation constraint has yet been obtained for any of these classes. Relying on pointwise extensibility, we obtain interpolation constraints for the class of weakly convex functions with bounded subgradient.
\section{Interpolation constraints for a class of weakly convex functions}\label{sec:weakconv}
\subsection{Motivation}  Relying on interpolation constraints and the PEP framework, a wide range of first-order methods has been tightly analyzed in the smooth (non)-convex setting, see e.g., \cite{rotaru2022tight,thesis,taylor2017smooth,taylor2018exact,taylor2019stochastic,de2017worst,dragomir2021optimal,guille2022gradient,kamri2023worst}. \com{However, despite interest in first-order methods applied to non smooth non convex functions, the lack of interpolation constraints in this setting prevents any a priori guarantee of tightness in the bounds obtained.} 

In particular, in this setting, widely analyzed functions are weakly convex functions, that is, functions convex up to the addition of a quadratic term \cite{nurminskii1973quasigradient}, and including convex and smooth functions, but also any composition $f(x):=h(c(x)),$
 where $h$ is convex and Lipschitz and $c$ is smooth \cite{drusvefficiency}. Weakly convex functions cover various large-scale applications, e.g., robust principal component analysis \cite{candes2011robust}, phase retrieval \cite{davis2020nonsmooth}, minimization of the conditional value-at-risk \cite{ben1986expected}, covariance matrix estimation \cite{chen2015exact} or training deep linear neural networks \cite{hardt2016identity}, while possessing convergence guarantees and a continuous measure of stationarity \cite{burke1985descent,drusvefficiency}. Since weak convexity alone does not guarantee the convergence of common algorithms such as the subgradient method, weakly convex functions are usually analyzed together with an additional assumption. In particular, we treat the case of weakly convex functions with bounded subgradients \cite{bohm2021variable,davis2019stochastic,das2024branch,mai2020convergence}.

\subsection{Weakly convex functions with bounded subgradient}
Weakly convex functions with bounded subgradients are commonly characterized by the following constraint $p_{\FuB}$
  \begin{align} \label{ClassicalFuB}
p_{\FuB}^{x,y}\geq 0 \Leftrightarrow
\begin{cases}
   & f(y)\geq f(x)+\langle g_x,y-x\rangle-\frac{\mu}{2}\|x-y\|^2\\
   &\|g_x\|\leq B,
    \end{cases}  
\end{align} that holds $\forall x,y \in \dom f$, and $\forall g_x\in\partial f(x)$. A computer-aided performance bound for the subgradient method applied to $\bF_{p_{\FuB}}(\R^d)$ was derived in \cite{das2024branch}, using $p_{\FuB}$. However, as illustrated in Section~\ref{sec:counterex}, $p_{\FuB}$ is not an interpolation constraint, hence, the bound obtained in \cite{das2024branch} is a priori not tight. 

\ARj{We now propose a strengthened version of $p_{\FuB}$, motivated by a graphical, one-dimensional intuition. We will then prove that this strengthened condition is an interpolation constraint, and that the class it defines is $\bF_{p_{\FuB}}(\R^d)$.}

\AR{Consider the one-dimensional case with parameters \( d = 1 \), \( \mu = 1 \), and \( B = 1 \), and the dataset \( S = \{(0, 0, 1),\ (3, -1.5, 1)\} \). It is straightforward to verify that this dataset satisfies \( p_{\FuB} \). However, as illustrated in Figure~\ref{fig:counterexample}, \( S \) cannot be interpolated by any weakly convex function with bounded subgradients.}

\begin{figure}[ht!]
    \centering
\includegraphics[width=\textwidth]{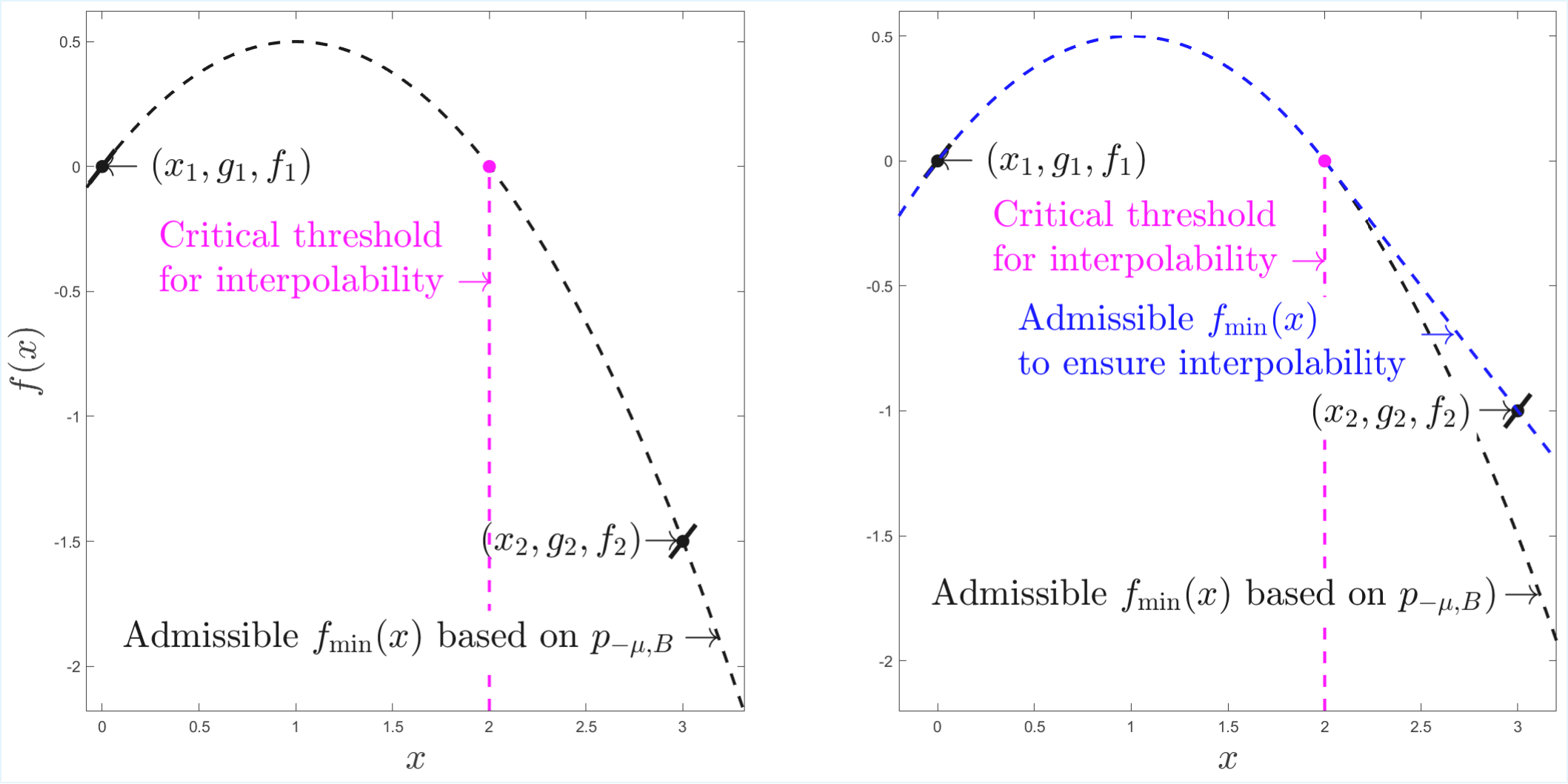}
    \caption{\AR{Counterexample to the interpolability of $p_{\FuB}$. On the left-hand side figure, the set $S=\{(x_i,f_i,g_i)\}_{i=1,2}$ satisfies \eqref{ClassicalFuB}, but cannot be interpolated by a function in $\bF_{p_{\FuB}}(\R^d)$, since the norm of the gradient associated to the only weakly convex function, i.e., whose gradient does not decrease faster than $- \|x\|$, interpolating $S$, $f_{\min}(x)$, grows larger than allowed from the critical threshold $x=2$. On the right-hand side figure, 
    the minimal function value allowed for $f_2$ ensures a bounded subgradient associated to the minimal weakly function interpolating $S$. An interpolation constraint for $\bF_{p_{\FuB}}(\R^d)$ should enforce this minimal value for $f_2$.}}
     \label{fig:counterexample}
\end{figure}

To strengthen $p_{\FuB}$, consider $\tilde p_{\FuB}$.
\begin{align}
\tilde{p}_{\FuB}^{x,y}\geq 0\Leftrightarrow
\begin{cases}
   &f(y)\geq f(x)+\langle g_x, y-x\rangle-\frac{\mu}{2}\|x-y\|^2+\frac{\mu}{2}\|y-C_{xy}\|^2\\
   &\|g_x\|\leq B,
    \end{cases}
    \label{22}
\end{align}
where $C_{xy}$ is the projection of $y$ onto $\mathcal{B}_{\frac{g(x)}{\mu}+x, \frac{B}{\mu}}$, i.e., $$C_{xy}=\underset{\|\mu (z-x)+g(x)\|\leq B}{\text{ argmin }} \|z-y\|.$$\vspace{-0.3cm}
\begin{remark}
   We call $\tilde{p}_{\FuB}$ a strengthening of $p_{\FuB}$ since it is stronger, i.e., satisfied by less triplets, than $p_{\FuB}$. Indeed, $\tilde{p}_{\FuB}$ differs from $p_{\FuB}$ of a positive term only.
\end{remark}

As illustrated on the right-hand side of Figure \ref{fig:counterexample}, $\tilde p_{\FuB}$ strengthens $p_{\FuB}$ by imposing the minimal function value allowed for $f_j$ with respect to some $(x_i,f_i,g_i)$ to always be compatible with a bounded subgradient. Observe that $C_{xy}$, the projection of $y$ onto $\mathcal{B}_{\frac{g(x)}{\mu}+x, \frac{B}{\mu}}$, \com{is the critical treshold for interpolability on Figure \ref{fig:counterexample}}. Whenever two points are sufficiently close, e.g., any $x_2\in (0,2]$ in Figure \ref{fig:counterexample}, then $C_{xy}=y$ and $\tilde p_{\FuB}$ reduces to $p_{\FuB}$. However, once this critical threshold $C_{xy}$ is reached, $\tilde p_{\FuB}$ imposes a larger minimal function value. We illustrate on Figure \ref{fig:wcadmissible} the difference between the admissible regions allowed by $p_{\FuB}$ and $\tilde p_{\FuB}$.
\begin{figure}[ht!]
    \centering
  \includegraphics[width=0.7\textwidth]{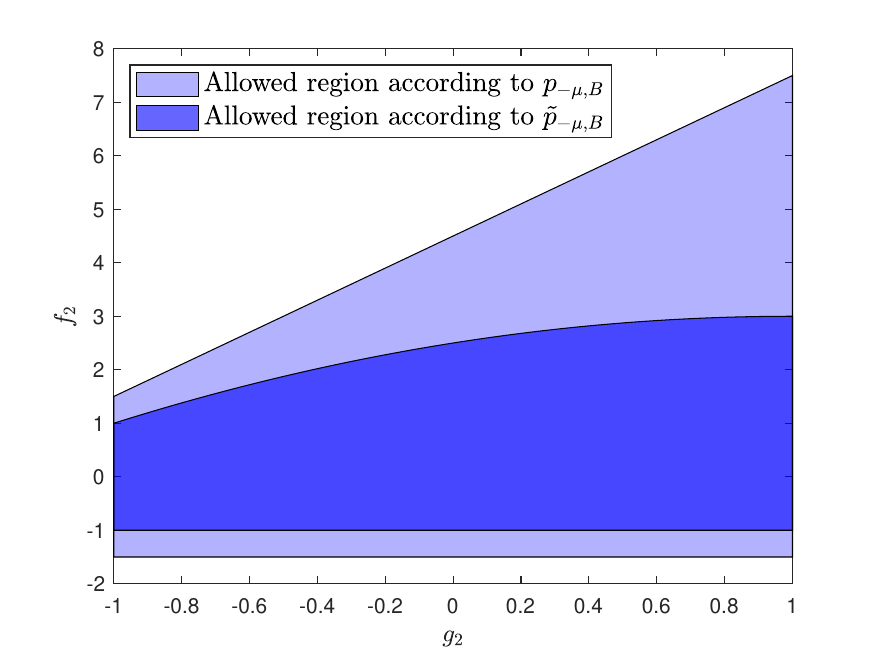}
    \caption{Given $\mu=B=1$, $(x_1,f_1,g_1)=(0,0,1)$, and $x_2=3$, allowed region for $f_2$ as a function of $g_2$, according to $p_{\FuB}$ and $\tilde p_{\FuB}$. Clearly, $p_{\FuB}$ is weaker than $\tilde p_{\FuB}$ since satisfied by triplets that are not $\bF_{p_{\FuB}}(\R^d)$-interpolable. \com{In particular, the maximal value of $f_2$ allowed differs by a factor 4.}} 
\label{fig:wcadmissible}
\end{figure} 

The main result of this section is the following. 
\begin{theorem}\label{thm:interp_WC_bounded_grad}
    Consider $\tilde p_{\FuB}$ as defined in \eqref{22}. Then, $p$ is an interpolation constraint for $\bF_{\tilde{p}_{\FuB}}(\R^d)=\bF_{p_{\FuB}}(\R^d)$, the class of weakly convex functions with bounded subgradients.
\end{theorem}

We proceed to the proof of Theorem \ref{thm:interp_WC_bounded_grad}, in two steps. We first show that $\tilde p_{\FuB}$ characterizes the class of weakly convex functions with bounded subgradients when imposed everywhere, before proving it is an interpolation constraint for this class. 
\paragraph{Equivalence of $\bF_{p_{\FuB}}(\R^d)$ and $\bF_{\tilde p_{\FuB}}(\R^d)$.}
\begin{proposition} \label{prop:WC_pUB_equiv_tpUB}
Consider $p_{\FuB}$ as defined in \eqref{ClassicalFuB} and $\tilde p_{\FuB}$ as defined in \eqref{22}. \com{It holds that $\bF_{p_{\FuB}}(\R^d)=\bF_{\tilde p_{\FuB}}(\R^d)$.}
\end{proposition} 

In the proof of Proposition \ref{prop:WC_pUB_equiv_tpUB}, we  will resort to the following Lemma.
 \begin{lemma} \label{leminterm}
Let $y_0$ be the projection onto $\mathcal{B}_{c,R}$ of $y \in \R^d \notin \mathcal{B}_{c,R}$, for some $c\in \R^d$ and $R\in \R_+$. Let \AR{$z=\alpha y+(1-\alpha) y_0$ for some $\alpha \in [0,1]$}. Then, $y_0$ is also the projection onto $\mathcal{B}_{c,R}$ of $z$. In addition, $\forall w \in \R^d$ such that $\|w\|\leq R$, 
  \begin{align*}
    \langle w, z-y\rangle \leq \langle c-y_0, z-y\rangle.
\end{align*}
\end{lemma} 
\begin{proof}
Note that $\forall x \in \mathcal{B}_{c,R}$, it follows from, e.g., \cite[Proposition 1.1.9]{bertsekas2009convex}, that $\langle x-y_0,z-y_0\rangle =\alpha \langle x-y_0,y-y_0\rangle\leq 0$. Hence, it holds that the projection of $z$ onto $\mathcal{B}_{c,R}$ equals $y_0$. \com{It also holds that for any $\|w\|\leq R$, 
    $\langle w, z-y\rangle\leq R\|z-y\|$, 
and $c-y_0$ reaches that since of norm $R$, parallel to, and of same direction as $z-y$. Hence, $c-y_0=\frac{R(z-y)}{\|z-y\|}$.}
\end{proof}

We now proceed to the proof of Proposition \ref{prop:WC_pUB_equiv_tpUB}.
\begin{proof}
The case $\mu=0$ and the inclusion $\bF_{\tilde p_{\FuB}}(\R^d) \subseteq \bF_{p_{\FuB}}(\R^d)$ follow from the definition, since the inequalities only differ by a positive term $\mu \|\cdot\|^2$. Hence, we consider $\mu>0$ and establish $\bF_{p_{\FuB}}(\R^d) \subseteq \bF_{\tilde p_{\FuB}}(\R^d)$. 

By contradiction, suppose there exists \AR{a mapping $(f,g)$} satisfying everywhere $p_{\FuB}$ but not $\tilde p_{\FuB}$. Then, at some $x,y\in \R^d$, given \AR{$f_x=f(x)$, $f_y=f(y)$, and for some $g_x\in g(x)$, $g_y\in g(y)$}, $p_{\FuB}$ is satisfied while $\tilde{p}_{\FuB}$ is violated, i.e., $$p_{\FuB}^{xy}\geq 0, \ p_{\FuB}^{yx}\geq 0, \ \tilde p_{\FuB}^{xy}< 0, \text{ w.l.o.g. in the choice of } x \text{ and }y.$$   \com{Since $p_{\FuB}^{xy}$ and $\tilde p_{\FuB}^{xy}$ only differ by $\mu \|C_{xy}-y\|^2$, this implies $C_{xy}\neq y$, where $C_{xy}$ is the projection of $y$ onto $\mathcal{B}_{\frac{g_x}{\mu}+x, \frac{B}{\mu}}$, and the existence of some $\beta_0 \in [0,1)$ such that
   \begin{align}
    &f_y=f_x+\langle g_x, y-x\rangle-\frac{\mu}{2}\|x-y\|^2+\frac{\beta_0 \mu}{2}\|C_{xy}-y\|^2,\label{assum3}
\end{align}
since $f_y$ is larger than the right-hand side with $\beta_0=0$ due to $p_{\FuB}^{xy}\geq 0$, and smaller than the right-hand side with $\beta_0=1$ due to $\tilde p_{\FuB}^{xy}<0$. }

\com{Consider the decreasing sequence $\{\beta_k\}_{k=0,1,...}$, initialized in \eqref{assum3} and defined by     \AR{\begin{align}
    \beta_{k+1}=\begin{cases}
        2-
    \frac{1}{\beta_k} & \beta_k >\frac{1}{2} \\
    0 & \beta_k\leq \frac{1}{2}
    \end{cases}
    \quad k=0,1,...,\label{recurs}
\end{align}} and the associated sequence $\{z_k\}_{k=0,1,...}$, initialized by $z_0=y$ and defined by
\AR{$$z_{k+1}=\begin{cases}
    \beta_kz_k+(1-\beta_k)C_{xy},& \beta_k\neq 0\\
    \frac{1}{2} (y+C_{xy})& \beta_k=0
\end{cases} \quad k=1,2,....$$}
Let $k=1$. Since $(f,g)$ satisfy $p_{\FuB}$ everywhere, letting $f_{z_1}=f(z_1)$, and for all $g_{z_1}\in g({z_1})$, it holds, 
   \begin{align}
\begin{split}
        &f_{z_1}\geq f_x +\langle g_x, {z_1}-x\rangle-\frac{\mu}{2}\|{z_1}-x\|^2:=f_{z_1}^-,  \\&
        f_{z_1}\leq f_y+\langle g_{z_1},{z_1}-y\rangle+\frac{\mu}{2}\|{z_1}-y\|^2:=f_{z_1}^+,\\& \|g_{z_1}\|\leq B.
    \label{interm1}
    \end{split}
\end{align}}
\com{\indent Since $C_{xy}$ is the projection of $y$ onto $\mathcal{B}_{\frac{g_x}{\mu}+x,\frac{B}{\mu}}$ and $z_1$ is a convex combination of $C_{xy}$ and $y$, by Lemma \ref{leminterm}, it holds 
  \begin{align}
    \langle g_{z_1}, {z_1}-y\rangle&\leq \langle g_x+\mu (x-C_{xy}), {z_1}-y\rangle.\label{ineq1}
\end{align}
In addition, by definition, ${z_1}-y=-K({z_1}-C_{xy})$, where $K=\frac{(1-\beta_0)}{\beta_0}$ if $\beta_0\neq 0$, $K=1$ otherwise. We have 
   \begin{align}
    \langle x-C_{xy}, {z_1}-y\rangle&=\langle z_1-C_{xy}, {z_1}-y\rangle+\langle x-z_1, {z_1}-y\rangle\nonumber\\
    &=-K\|{z_1}-C_{xy}\|^2+\frac{\|x-y\|^2}{2}-\frac{\|z_1-x\|^2}{2}-\frac{\|z_1-y\|^2}{2}.\label{ineq2}
\end{align}
Consequently, 
   \begin{align*}
f_{z_1}^+&\underset{\eqref{ineq1}}{\leq } f_y+\langle g_x,{z_1}-y\rangle +\mu \langle x-C_{xy}, {z_1}-y\rangle+\frac{\mu\|z_1-y\|^2}{2}
\\
&\underset{\eqref{ineq2}}{=} f_y+\langle g_x,{z_1}-y\rangle-K\|{z_1}-C_{xy}\|^2+\frac{\mu \|x-y\|^2}{2}-\frac{\mu\|z_1-x\|^2}{2}
\\&\underset{\eqref{assum3}}{=}f_x+\langle g_x,{z_1}-x\rangle-\frac{\mu}{2}\|{z_1}-x\|^2+\frac{\mu\beta_0}{2}\|C_{xy}-y\|^2-K\| C_{xy}-{z_1}\|^2\end{align*}}

\AR{If $\beta_0=0$, this implies $f_{z_1}^+=f_{z_1}^- -\| C_{xy}-{z_1}\|^2<f_{z_1}^-$, hence the assumption "$(f,g)$ satisfy $p_{\FuB}$ everywhere" cannot hold at $z_1$.} Else, 
\begin{align*}
    f_{z_1}^+&\underset{C_{xy}-y=\frac{1}{\beta_0}( C_{xy}-{z_1})}{=}f_{z_1}^-+\frac{\mu}{2}\frac{2\beta_0-1}{\beta_0}\| C_{xy}-{z_1}\|^2=f_{z_1}^-+\frac{\mu}{2}\frac{2\beta_0-1}{\beta_0}\|C_{xz_1}-z_1\|^2,
\end{align*}
where the last equality follows from the first assertion in Lemma \ref{leminterm}. By definition of $\beta_1$, $f_{z_1}^-$ and $f_{z_1}^+$, it then holds
   \begin{align}
    f_{z_1}= f_x+\langle g_x,{z_1}-x\rangle-\frac{\mu}{2}\|{z_1}-x\|^2+\frac{\beta_1\mu}{2}\|C_{x{z_1}}-{z_1}\|^2,\label{assum3new}
\end{align} i.e., \eqref{assum3} with $z_1$ and $\beta_1$ replacing $z_0$ and $\beta_0$. Starting from \eqref{assum3new} and repeating the argument with $k=2,3,...$, it holds
   \begin{align}
    f_{z_k}= f_x+\langle g_x,{z_k}-x\rangle-\frac{\mu}{2}\|{z_k}-x\|^2+\frac{\beta_k\mu}{2}\|C_{x{z_k}}-{z_k}\|^2, \ k=1,2,...,\label{assum3newnew}
\end{align}
where $\beta_k$ is defined as in \eqref{recurs}.
The sequence $\{\beta_k\}_{k=1,2,...}$ can be equivalently defined as  \AR{\begin{align}
    \beta_k=\begin{cases}
        \frac{(k+1)\beta_0-k}{k\beta_0-(k-1)}& k<\frac{2\beta_0-1}{1-\beta_0}\\
        0 & k\geq \frac{2\beta_0-1}{1-\beta_0}
    \end{cases}, \quad k=0,1,....\label{recurs1}
\end{align} }
\AR{Hence, at any $z_{k+1}$ where $k\geq \frac{2\beta_0-1}{1-\beta_0}$, following the same argument as when $\beta_0=0$, replacing $\beta_0$ by $\beta_k=0$ and $z_1$ by $z_{k+1}$, the assumption "$(f,g)$ satisfy $p_{\FuB}$ everywhere" cannot hold, which completes the proof.}
\end{proof}
\paragraph{Interpolability of $\tilde p_{\FuB}$.}
We now prove interpolability of $\tilde{p}_{\FuB}$, using pointwise extensibility. 
 \begin{proposition}\label{prop:WC_tpUB_is_interp}
     Constraint $\tilde{p}_{\FuB}$, defined in \eqref{22}, is an interpolation constraint for $\bF_{\tilde{p}_{-\mu,B}}(\R^d)$.
 \end{proposition}
\begin{proof}
Consider $S=\{(x_i,f_i,g_i\}_{i\in [N]}$ satisfying $p$, and any $x\in \R^d$. We seek for some $f \in \mathbb{R}$, $g \in \mathbb{R}^d$ extending $S$ to $x$, that is
   \begin{align}
    &f\geq f_i+\langle g_i,x-x_i\rangle-\frac{\mu}{2}\|x-x_i\|^2+\frac{\mu}{2}\|x-C_{i+}\|^2, \quad \forall i \in [N] \label{B1} \\ 
    &f_j \geq f+\langle g,x_j-x\rangle-\frac{\mu}{2}\|x-x_j\|^2+\frac{\mu}{2}\|x_j-C_{+j}\|^2:=f_j^- , \quad \forall j \in [N]\label{B2}\\ 
    &\|g\|\leq B, \label{B3}  
\end{align}
where $C_{i+}$ denotes the projection of $x$ onto $\mathcal{B}_{\frac{g_i}{\mu}+x_i,\frac{B}{\mu}}$ and $C_{+j}$ denotes the projection of $x_j$ onto $\mathcal{B}_{\frac{g}{\mu}+x,\frac{B}{\mu}}$. Let  
   \begin{align*}
i^*&=\underset{i\in [N]}{\text{argmax }} f_i+\langle g_i,x-x_i\rangle-\frac{\mu}{2}\|x-x_i\|^2+\frac{\mu}{2}\|x-C_{i+}\|^2,\\f&=f_{i^*}+\langle g_{i^*},x-x_{i^*}\rangle-\frac{\mu}{2}\|x-x_{i^*}\|^2+\frac{\mu}{2}\|x-C_{i^*+}\|^2,\\
g&= g_{i^*}+\mu (x_{i^*}-C_{i^*+}).
\end{align*}
Then, $f$ satisfies \eqref{B1}, and  by definition of $C_{i^*+}$, $g$ satisfies \eqref{B3}. It remains to \com{show \eqref{B2} is satisfied, i.e., $\forall j\in [N] $, $f_j\geq f_j^-$. Given the expression of $f$ and $g$, $f_j^-$ becomes}
   \begin{align*}
    f_j^-&= f_{i^*}+\langle g_{i^*},x_j-x_{i^*}\rangle+\mu\langle x_{i^*}-C_{i^*+},x_j-x\rangle\\&\quad-\frac{\mu}{2}\|x-x_j\|^2+\frac{\mu}{2}\|x_j-C_{+j}\|^2-\frac{\mu}{2}\|x-x_i^*\|^2+\frac{\mu}{2}\|x-C_{i^*+}\|^2\\&
    =f_{i^*}+\langle g_{i^*},x_j-x_{i^*}\rangle-\frac{\mu}{2}\|x_{i^*}-x_j\|^2+\frac{\mu}{2}\|x_j-C_{+j}\|^2+\frac{\mu}{2}\|x-C_{i^*+}\|^2\\&\quad+\mu\langle x-C_{i^*+},x_j-x\rangle.
\end{align*}
Observe that, by definition of $g$, $C_{i^*j}+x-C_{i^*+}\in \mathcal{B}_{\frac{g}{\mu}+x,\frac{B}{\mu}}$, implying    \begin{align*}\|x_j-C_{+j}\|^2&\leq \|x_j-C_{i^*j}-x+C_{i^*+}\|^2\\&=\|x_j-C_{i^*j}\|^2+\|x-C_{i^*+}\|^2 -2\langle x_j-C_{i^*j}, x-C_{i^*+} \rangle
\\
\Leftrightarrow \|x-C_{i^*+}\|^2+\|x_j-C_{+j}\|^2&\leq \|x_j-C_{i^*j}\|^2+2\|x-C_{i^*+}\|^2\\&\quad-2\langle x_j-C_{i^*j}, x-C_{i^*+} \rangle\\&=\|x_j-C_{i^*j}\|^2-2\langle x_j-C_{i^*j}+C_{i^*+}-x, x-C_{i^*+}\rangle.
\end{align*} 
Hence, 
 \begin{align}
\begin{split}
f_j^-&\leq f_{i^*}+\langle g_{i^*},x_j-x_{i^*}\rangle-\frac{\mu}{2}\|x_{i^*}-x_j\|^2+\frac{\mu}{2}\|x_j-C_{i^*j}\|^2\\&\quad+\mu\langle C_{i^*j}-C_{i^*+}, x-C_{i^*+}\rangle\\
&\leq f_{i^*}+\langle g_{i^*},x_j-x_{i^*}\rangle-\frac{\mu}{2}\|x_{i^*}-x_j\|^2+\frac{\mu}{2}\|x_j-C_{i^*j}\|^2,\label{lastbrick}
\end{split}
\end{align}\vspace{-0.3cm}
where the last inequality follows from the fact that $C_{i^*j}\in \mathcal{B}_{\frac{g_{i^*}}{\mu}+x_{i^*},\frac{B}{\mu}}$ and that for any closed set $C$, $x\in \R^d$, $z \in C$, and $P_x$ the projection of $x$ onto $C$, it holds that $\langle x-P_x, P_x-z\rangle \geq 0$. Finally, observe that $\forall j =[N]$,
   \begin{align*}
f_j&\underset{p^{ji^*}\geq 0}{\geq} f_{i^*}+\langle g_{i^*},x_j-x_{i^*}\rangle-\frac{\mu}{2}\|x_{i^*}-x_j\|^2+\frac{\mu}{2}\|x_j-C_{i^*j}\|^2\underset{\eqref{lastbrick}}{\geq} f_j^-,
\end{align*}
so that \eqref{B2} is satisfied and 
$\tilde{p}_{\FuB}$ is pointwise extensible. \AR{In addition, $\tilde{p}_{\FuB}$ is continuous and differentially consistent since $-\frac \mu 2 x^2 =\omega(x)
$, and $\|g\|\leq B$, i.e., $g$ is regularly extensible since it satisfies \eqref{eq:reg1}, see Lemma \ref{lem:reg_ext_example}.} By Theorem \ref{theoequi}, it is an interpolation constraint for $\bF_{\tilde{p}_{\FuB}}(\R^d)$.
\end{proof}
Combining Propositions \ref{prop:WC_pUB_equiv_tpUB} and \ref{prop:WC_tpUB_is_interp} proves Theorem \ref{thm:interp_WC_bounded_grad}.
\paragraph{Performance analysis of the subgradient method on $\bF_{p_{\FuB}}(\R^d)$.} We now show how relying on $\tilde p_{\FuB}$ allows improving on performance guarantees for the subgradient method on $\bF_{p_{\FuB}}(\R^d)$, whose steps are given by:
   \begin{align*}
    x_{k+1}=x_k-h g_k, \ g_k \in \partial f(x_k), \ k=0,1,...,N.
\end{align*}

In the non smooth non convex setting, since usual criteria of efficiency can vary discontinuously \cite{davis2019stochastic,nurminskii1973quasigradient}, one classically relies on the \emph{Moreau envelope} \cite{moreau}, defined as
   \begin{equation}
    f_{\rho }(x):=\underset{y \in \R^d}{\text{min }}{f(y)+\frac{\|x-y\|^2}{2\rho}} \  \forall x\in \R^d, 
\end{equation}
and takes the norm of its gradient as a near-stationarity measure, see \cite{davis2019stochastic} for more details. As in \cite{das2024branch}, we thus choose as performance measure    \begin{align*}
    &\mathcal{E}=\frac{1}{N+1}\sum_{i=0}^N\|\nabla f_{\rho}(x_i)\|^2=\frac{1}{N+1}\sum_{i=0}^N \frac{\|x_i-\textbf{prox}_{\rho f}(x_i)\|^2}{\rho^2},\\
    \text{where }\quad &\textbf{prox}_{\rho f}(x):=\underset{y \in \R^d}{\text{argmin }}{f(y)+\frac{\|x-y\|^2}{2\rho}}\ \forall x\in \R^d,
\end{align*} and $x_0,x_1,...$ are the iterates of the method. As initial condition, we impose $f_{\rho }(x_0)-f^*\leq R^2$. \com{To derive the exact worst-case performance of the subgradient method after $N$ iterations on $\bF_{p_{\FuB}}(\R^d)$, we follow the PEP framework introduced in Appendix \ref{app:PEP}, and formulate the search of a worst-case function in $\bF_{\tilde p_{\FuB}}(\R^d)$ as an optimization problem over a set $S$ of iterates satisfying $\tilde p_{\FuB}$}, i.e.,
   \begin{equation}
     \begin{aligned}
   &  && \underset{\substack{S=\{(x_i,f_i,g_i)\}_{i\in [N]\cup*}\\ S_y=\{(y_i,f^y_i,g^y_i)\}_{i\in [N]} }}{\max\ } \frac{1}{N+1}\sum_{i=0}^N \frac{\|x_i-y_i\|^2}{\rho^2}\label{PEP_wc}\\
     &\text{Subgradient method:}&&\text{ s.t. } x_{i+1}=x_i-h g_i, \quad i\in [N-1]\\& y_i=\textbf{prox}_{\rho f}(x_i):
     &&\quad \quad y_i=x_i-\frac{1}{\rho}g_i^y, i\in [N]\\
     &S\cup S_y\text{ is } \bF_{\tilde p_{\FuB}(\R^d)}\text{-interpolable: }&&\quad \quad \tilde p^{ij}_{\FuB}\geq 0 \ \forall i ,j\in S\cup S_y\\
     &\text{Initial condition:}&&\quad \quad f^y_0-f^*+ \AR{\frac{\|x_0-y_0\|^2}{2\rho}} \leq R^2 \\
     &\text{Optimality:}&&\quad \quad g^*=0,\ f_i\geq f^*,\  f_i^y\geq f^*,\  i\in [N], 
\end{aligned}
\end{equation}
where the interpolation conditions impose $f_i=f(x_i)$, $g_i\in\partial f(x_i)$, $f_i^y=f(y_i)$, $g^y_i\in \partial f(y_i)$ for a single function $f \in \bF_{\tilde p_{\FuB}}$.
 \begin{remark}
    The bound obtained in \cite{das2024branch} results from a similar problem to \eqref{PEP_wc}, with two main differences. First, a single iteration is analyzed, to derive a Lyapounov  function valid for all iterations. Then, the constraint imposed on $S\cup S_y$ is $p_{\FuB}$ and not $\tilde p_{\FuB}$. The improvement in the bound we obtain thus arises from (i) relying on all information available instead of a single iteration, and (ii) the benefit of relying on interpolation constraints, even when being able to optimally combine constraints. 
\end{remark}

\com{In Appendix \ref{app:weaklyconvprog}, we show that Problem \eqref{PEP_wc} is tractable since linear in the function values and scalar products of $x_i$, $g_i$ and $C_{ij}$, and hence amenable to SDP formulation.}

Letting $\rho=2$ as in \cite{das2024branch}, we compare on Figure \ref{fig1} the bound resulting from \eqref{PEP_wc} with the bounds obtained (i) via a classical analysis \cite{davis2019stochastic} and (ii) via PEP \cite{das2024branch}, both relying on $p_{\FuB}$, as summarized in Table \ref{tab:summary}. \com{The analytical expression from \cite{das2024branch} was identified based on the numerical results obtained via PEP, which are exact bounds, and afterwards proved analytically. \begin{remark}\label{rem:minN}
    By definition of $\bF_{\tilde p_{\FuB}}(\R^d)$ and $\mathcal{E}$, it holds that the performance measure is at most of $B^2$. Hence, bounds on $\mathcal{E}$ are trivial whenever larger than $B^2$.
\end{remark}}
\begin{table}[ht]
    \centering
    \renewcommand{\arraystretch}{1.3}
    \begin{tabular}{@{}llll@{}}
    \toprule
    \textbf{Type of analysis} & \textbf{Optimal $h$} & \textbf{Bound on $\mathcal{E}$} & \textbf{Minimal $N$}\\ \toprule
        Conditions $p_{\FuB}$,&$\frac{R\mu}{B\sqrt{N+1}}$ &   $            \frac{4BR}{\mu \sqrt{N+1}}$ & $N\geq (\frac{4R}{\mu B})^2-1$  \\
        classical analysis \cite{davis2019stochastic}&& &  \\\hline
         Conditions $p_{\FuB}$,& $\frac{\sqrt{4\frac{R\mu}{B}(N+1)+1}}{2(N+1)}$&    $B^2\frac{2\sqrt{4\frac{R^2\mu^2}{B^2}(N+1)+1}-1}{\mu^2(N+1)}$& $N\geq (\frac{2R}{\mu B})^4-1$\\
          PEP-based analysis \cite{das2024branch}&  & &\\\bottomrule
    \end{tabular}    
    \caption{\com{Existing bounds on the performance of the subgradient method on $\bF_{\tilde p_{\FuB}}(\R^d)$, and associated stepsizes minimizing these bounds. The minimal $N$ is the value for which the bounds becomes non trivial, see Remark \ref{rem:minN}.}}
    \label{tab:summary}
\end{table} 

\begin{figure}[ht!]
    \centering
    \begin{subfigure}{0.49\textwidth}
        \centering
\includegraphics[width=\textwidth]{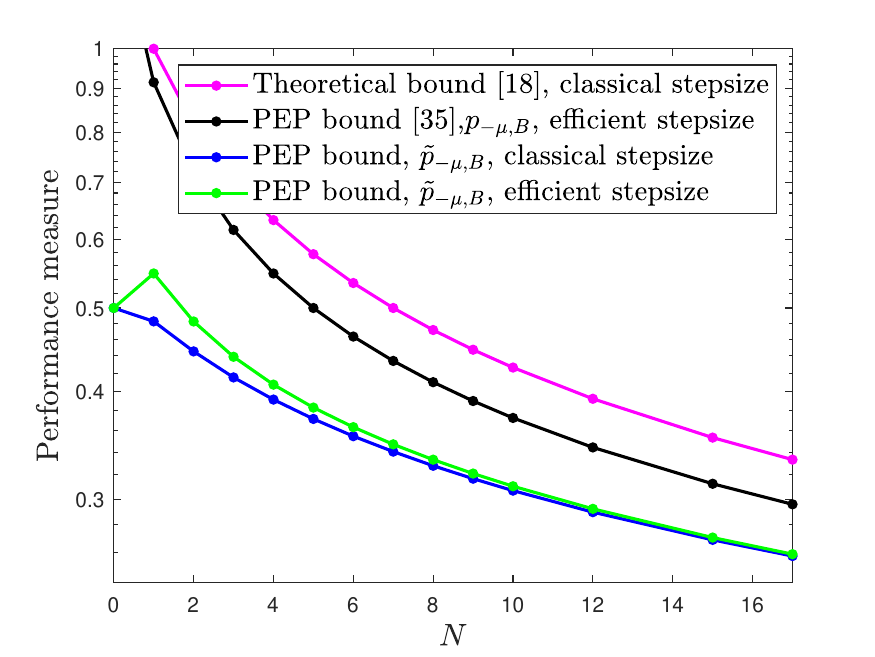}
    \end{subfigure}%
    ~ 
    \begin{subfigure}{0.49\textwidth}
        \centering
\includegraphics[width=\textwidth]{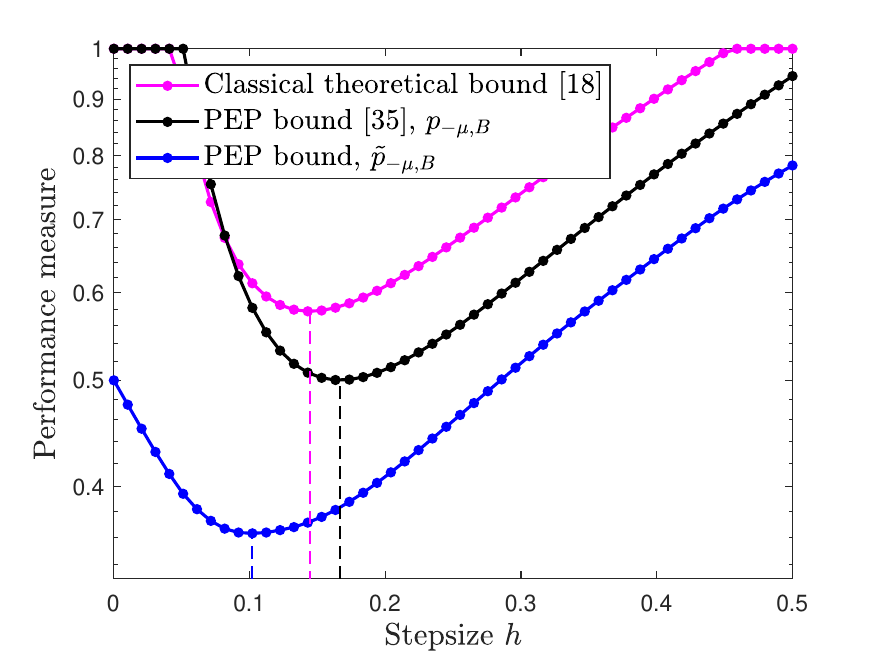}
    \end{subfigure}
    \caption{Performance of the subgradient method on $\bF_{\tilde p_{\FuB}}$ for $R^2=\frac{1}{8}$ and $B,L=1$. On the left-hand side figure, the PEP bounds obtained relying on $\tilde p_{\FuB}$, with both choices of stepsizes (obtained by minimizing existing bounds) outperform the state of the art ones, and are always non-trivial. In addition, the classical stepsize actually yields slightly better results than the efficient one. On the right-hand side figure, the bounds obtained after $5$ iterations are compared with respect to different choices of stepsizes $h$. Optimizing the PEP-based bound relying on $p_{\FuB}$ leads to an unoptimal choice of $h$.}
\label{fig1}
\end{figure}

Solving Problem \eqref{PEP_wc}, we outperform both bounds with both choices of the stepsize $h$, as illustrated on Figure \ref{fig1}. We also observe that in the exact setting, the efficient stepsize as computed in \cite{das2024branch} turns out to be less efficient than the one derived in \cite{davis2019stochastic}, highlighting the necessity to not only optimally combining constraints representing function classes, but also to rely on interpolation constraints to efficiently tune methods, and illustrating the possible drawbacks of optimizing steps based on conservative bounds. 
 \begin{remark}
Note that pointwise extensibility allowed verifying whether or not $\tilde p_{\FuB}$ was an interpolation constraint by straightforwardly manipulating algebraic inequalities. The difficulties lied in proving $\bF_{\tilde p_{\FuB}}(\R^d)=\bF_{p_{\FuB}}(\R^d)$, which further motivates the proposed constraint-based approach.  \newpage
\end{remark}
\section{Exhaustive analysis of a class of constraints}\label{sec:gram}
Motivated by the difficulty to obtain equivalent characterizations of given a function class, and by the observation that function classes considered as interesting in optimization do not always have a simple tight description, see Section \ref{sec:weakconv}, we illustrate the constraint-based approach to interpolation presented in Section \ref{sec:interp} on a simple class of constraints. We consider constraints consisting of a single inequality, linear in function values and scalar product of points and subgradients, \AR{and such that $p^{xx}=0$, i.e., any single data point satisfies $p$.} Hence, we analyze constraints \begin{align}\label{generalConstraint}
    p^{ij}\geq 0 \Leftrightarrow f_i&\geq f_j+B\|g_i\|^2+C\|g_j\|^2+D\langle g_j,g_i\rangle +E\|x_i\|^2+F\|x_j\|^2\nonumber\\&+G\langle x_i,x_j\rangle +H \langle g_i,x_i\rangle +I \langle g_i,x_j\rangle+J \langle g_j,x_i\rangle+K \langle g_j,x_j\rangle,
\end{align}
where $D=-(B+C)$, $G=-(E+F)$ and $K=-(H+I+J)$. Indeed, such constraints fit the PEP framework, and tightly describe several first-order function classes, for given parameters. For instance, $\forall -\infty < \mu<L< +\infty$,  \begin{align}\label{ClassicalmuL}
       p^{ij}\geq 0 \Leftrightarrow f_i\geq f_j+\frac{\langle Lg_j-\mu g_i,x_i-x_j\rangle}{L-\mu}+\frac{\|g_i-g_j\|^2}{2(L-\mu)}+\frac{\mu L \|x_i-x_j\|^2}{2(L-\mu)},
\end{align} is an interpolation constraint characterizing $\mu$-strongly convex $L$-smooth functions \cite[Theorem 4]{taylor2017smooth}.

We show that the only differentially consistent interpolation constraints of the form \eqref{generalConstraint} are \eqref{ClassicalmuL}, and \com{a constraint describing the class of quadratic functions $f(x)=\frac{\mu}{2}\|x\|^2+bx+c$ for some $b\in\R^d$, $c\in \R$, i.e., the limit case of \eqref{ClassicalmuL}}. Equivalently, these are the only function classes described by \eqref{generalConstraint} that admit a simple tight characterization.
\begin{remark}
    \AR{The second function class is in itself trivial, since it can be exactly characterized by equalities $g_i=\mu x_i+b$, $f_i=\frac{\mu}{2}\|x_i\|^2+bx+c$. However, we present its characterization via a constraint of the form \eqref{generalConstraint} (i) as a part of the exhaustive analysis we conduct, and (ii) to characterize the whole spectrum $-\infty< \mu \leq L<+\infty$  by a set of similar inequalities, without having to distinguish $\mu=L$ as a special case.}
\end{remark}
 \begin{theorem}\label{thm:extensive_analysis_results}
    Let $p$ be of the form \eqref{generalConstraint}. Then $p$ is (i) differentially consistent, and (ii) an interpolation constraint if and only if $p^{ij}\geq 0$ can be expressed as in \eqref{ClassicalmuL}, where $\mu<L$ and $\mu, L \in (-\infty, \infty)$, or as 
      \begin{align}
    p^{ij}\geq 0 \Leftrightarrow f_i\geq f_j+\frac{1}{2}\langle g_i+g_j,x_i-x_j\rangle+ M\|g_i-g_j-\mu(x_i-x_j)\|^2,\label{youpie}
\end{align}
where $M > 0$. \com{In the first case, $\bF_p(\R^d)$ is the class of $\mu$-(strongly/weakly)-convex $L$-smooth functions, and in the second case, $\bF_p(\R^d)$ is the class of quadratic functions $f(x)=\frac{\mu}{2}\|x\|^2+bx+c$ for some $b\in\R^d$, $c\in \R$.}
\end{theorem}
\com{We first show that \eqref{ClassicalmuL} and \eqref{youpie} are indeed differentially consistent interpolation constraints. }
  \begin{proposition}[{\cite[Theorem 4]{taylor2017smooth}}]\label{propipropa}
    Consider a constraint $p$ defined by \eqref{ClassicalmuL}. Then $p$ is a differentially consistent interpolation constraint, and the class $\bF_p(\R^d)$ it defines is the class of $\mu$-(strongly/weakly) convex $L$-smooth functions.
    \end{proposition}
      \begin{proposition}
    Consider a constraint $p$ defined by \eqref{youpie}. Then $p$ is a differentially consistent interpolation constraint, and the class $\bF_p(\R^d)$ it defines is the class of quadratic functions of the form $f(x)=\frac{\mu}{2}\|x\|^2+bx+c$ for some $b\in\R^d$, $c\in \R$..
    \end{proposition}
    \begin{proof}
    First, note that $p^{ij}\geq 0$ implies
     \begin{align*}
        f_i&\geq f_j+\langle g_j,x_i-x_j\rangle+M \|g_i-g_j+\frac{1-4M\mu}{4M}(x_i-x_j)\|^2\\&\quad+(M\mu^2-\frac{(1-4M\mu)^2}{16M})\|x_i-x_j\|^2\\&\geq f_j+\langle g_j,x_i-x_j\rangle+\omega(\|x_i-x_j\|).
    \end{align*}
    Hence, $p^{ij}\geq 0$ implies \eqref{eq:subdiff_gx} and $p$ is differentially consistent.
    
Consider now $S=\{(x_i,f_i,g_i)\}_{i\in [N]}$ satisfying $p$, and any $x\in \R^d$. For any $i,j\in [N]$, summing $p^{ij}\geq 0$ and $p^{ji}\geq 0$ implies
          \begin{align}\label{coolegalite}
    &0\geq \|(g_i-g_j)-\mu (x_i-x_j) \|^2 \Leftrightarrow g_i-g_j=\mu(x_i-x_j), 
\end{align}
which in turn implies, by combining $p^{ij}\geq 0$ and $p^{ji}\geq 0$,
  \begin{align}\label{coolegalite1}
    &f_i=f_j+\frac{1}{2}\langle g_i+g_j,x_i-x_j\rangle.
\end{align}
We let $g=g_i+\mu(x-x_i)$ for an arbitrary $i\in [N]$, which implies by \eqref{coolegalite} $g=g_j+\mu (x-x_j), \ \forall j \in[N]$. Then, $M\|g-g_i-\mu(x-x_i)\|^2=0$, and extending $p$ to $x$ amounts to obtaining a $f$ such that $f= f_i+\frac{1}{2}\langle g+g_i,x-x_i\rangle,\ \forall i \in [N]$. Let
  \begin{align*}
    f&=f_i+\frac{1}{2}\langle g+g_i,x-x_i\rangle\\&\underset{\eqref{coolegalite1}}{=} f_j+\frac{1}{2}\langle g+g_j,x-x_j\rangle+\frac{1}{2}\left(\langle g,x_j-x_i\rangle+\langle g_i,x-x_j\rangle+\langle g_j,x_i-x\rangle\right)\\&
    \underset{\eqref{coolegalite}}{=}
    f_j+\frac{1}{2}\langle g+g_j,x-x_j\rangle\\&\quad \quad \quad +\frac{1}{2}\left(\langle g,x_j-x_i\rangle+\langle g-\mu(x-x_i),x-x_j\rangle+\langle g-\mu(x-x_j),x_i-x\rangle\right)\\
    &=f_j+\frac{1}{2}\langle g+g_j,x-x_j\rangle \ \forall i,j\in [N].
\end{align*}
This choice of $f$ and $g$ extends $p$ to $x$, hence $p$ is pointwise extensible. Since $\|g\|\leq\|g_i\|+\mu \|x-x_i\|$, $\forall i\in [N]$, and $p$ is continuous and differentially consistent, $p$ is also regularly pointwise extensible, since it satisfies \eqref{eq:reg}, see Lemma \ref{lem:reg_ext_example}. Hence, by Theorem \ref{theoequi}, $p$ is an interpolation constraint.

Finally, it holds that $\bF_p(\R^d)$ is the class of quadratic functions $f=\frac{\mu}{2}\|x\|^2+bx+c$. Indeed, consider any $f \in \F_p(\R^d)$. Then, by \eqref{coolegalite} and \eqref{coolegalite1}, it holds $\forall x,y\in \R^d$ and $g(x)=\nabla f(x)$, 
 \begin{align*}
    g(x)&=g(y)+\mu(x-y) \quad \quad \quad \ \Leftrightarrow g(x)=\mu x+b,\\
    f(x)&=f(y)+\frac{\mu}{2} (\|x\|^2-\|y\|^2) +b(x-y)\Leftrightarrow f(x)=\frac{\mu}{2}\|x\|^2+bx+c,
\end{align*} for some $b\in\R^d,\ c\in \R$. Suppose now $f(x)=\frac{\mu}{2}\|x\|^2+bx+c$. Then, $\forall x,y \in \R^d$, $f$ satisfies $p$, hence $f\in \bF_p(\R^d)$.
\end{proof}

We now prove that no other constraint of the form \eqref{generalConstraint} is a differentially consistent interpolation constraint. First, we analyze differential consistency of \eqref{generalConstraint}.
     \begin{proposition}
    Consider a constraint $p$ of the form \eqref{generalConstraint}. Then $p$ is differentially consistent if and only if $p^{ij}\geq 0$ can be expressed as
              \begin{align}\label{consist}
     f_i\geq& f_j+\langle \gamma  g_i+(1-\gamma) g_j,x_i-x_j\rangle+\alpha\|g_i-g_j\|^2+\beta\|x_i-x_j\|^2,
\end{align}
where $\alpha\geq 0$.
\label{propchiante}
    \end{proposition}
The proof of Proposition \ref{propchiante} is deferred to Appendix \ref{app:propchiante}. Second, we show that if $p$ is of the form \eqref{consist}, but cannot be expressed as in Theorem \ref{thm:extensive_analysis_results}, then $p$ is not an interpolation constraint.
  \begin{proposition}
    Consider a constraint $p$ of the form \eqref{consist}. If $p$ cannot be expressed as \eqref{ClassicalmuL} or \eqref{youpie}, then $p$ is not pointwise extensible.
   \end{proposition}
   \begin{proof}
   Let $\alpha=0$. Then, \eqref{consist} can be written as \eqref{ClassicalmuL} or \eqref{youpie} if and only if $\gamma=0$ (corresponding to $L=\infty$) or $\gamma=1$ (corresponding to $\mu=-\infty$).  In Table \ref{tab:alpha0}, we provide counterexamples to the extensibility of \eqref{consist} for any other $\gamma$. 
     \begin{table}[H]
    \renewcommand{\arraystretch}{1.3}
    \begin{tabular}{@{}lll@{}}
    \toprule
    \textbf{$\gamma$}& \textbf{Initial data set} & \textbf{Added point} $x$
    \\\hline
    $\gamma<0$&$S=\bigg\{\begin{pmatrix}
        0\\0\\1
    \end{pmatrix},\begin{pmatrix}
        -1\\ \frac{\beta}{1-2\gamma}-1+\gamma\\\frac{-2\beta}{1-2\gamma}
    \end{pmatrix},\begin{pmatrix}
        -1\\\frac{\beta}{1-2\gamma}-1+\gamma\\\frac{-2\beta}{1-2\gamma}
    \end{pmatrix}\bigg\}$ & $x=-0.5$\vspace{1mm}\\ 
         $\gamma>1$&$S=\bigg\{\begin{pmatrix}
        0\\0\\1
    \end{pmatrix},\begin{pmatrix}
        -1\\\frac{\beta}{1-2\gamma}-1-\gamma\\\frac{-2\beta}{1-2\gamma}+2
    \end{pmatrix}\bigg\}$ & $x=-0.5$\vspace{1mm}\\
         $0<\gamma\leq \frac{1}{2}$&$S=\bigg\{\begin{pmatrix}
        0\\0\\1
    \end{pmatrix},\begin{pmatrix}
        -1\\\frac{\beta}{1-2\gamma}-1+\gamma\\\frac{-2\beta}{1-2\gamma}
    \end{pmatrix},\begin{pmatrix}
        -1.1\\\frac{\beta}{(1-2\gamma)}-1+\gamma\\\frac{-2\beta}{1-2\gamma}+1
    \end{pmatrix}\bigg\}$ & $x=-0.5$ \vspace{1mm}
         \\
         $\frac{1}{2}<\gamma<1 $&$S=\bigg\{\begin{pmatrix}
        0\\0\\1
    \end{pmatrix},\begin{pmatrix}
        -1\\\frac{\beta}{1-2\gamma}-1-\gamma\\\frac{-2\beta}{1-2\gamma}+2
    \end{pmatrix},\begin{pmatrix}
        -1.1\\\frac{\beta}{(1-2\gamma)}-1-\gamma\\ \frac{-2\beta}{1-2\gamma}+1
    \end{pmatrix}\bigg\}$
         & $x=-0.5$ \vspace{1mm}\\ \bottomrule
    \end{tabular}
    \caption{Numerical counterexamples to the pointwise extensibility of \eqref{consist} when $\alpha=0$, $\gamma \neq 0,1$.}
    \label{tab:alpha0}
\end{table}

Let $\alpha>0$, and consider $\Delta=\frac{\beta}{\alpha}+\frac{(1-\gamma)\gamma}{4\alpha^2}$. Then, \eqref{consist} can be written as \eqref{ClassicalmuL} or \eqref{youpie} if and only if $\Delta=0$, leading to \eqref{ClassicalmuL} or $\Delta=\frac{1}{16\alpha^2}$, leading to \eqref{youpie}. In Table \ref{tab:alpha}, we provide counterexamples to the interpolability of any other choice of $\Delta$.

  \begin{table}[ht!]
    \renewcommand{\arraystretch}{1.3}
    \begin{tabular}{@{}lll@{}}
    \toprule
    \textbf{$\beta$}& \textbf{Initial data set}  & \textbf{Added point} $x$
    \\ \hline 
    $\beta<0$&$S=\bigg \{\begin{pmatrix}
        0\\0\\0
    \end{pmatrix},\begin{pmatrix}
        1\\-\beta+(1-\gamma)\frac{1-2\gamma}{4\alpha})-(\frac{1-2\gamma}{4\alpha}))^2 \\\frac{1-2\gamma}{4\alpha}
    \end{pmatrix}\bigg\}$ & $x=0.5$ \vspace{1mm}\\  
         $\beta>\frac{1}{16\alpha
^2}$&$S=\bigg \{\begin{pmatrix}
        0\\0\\0
    \end{pmatrix},\begin{pmatrix}
        1\\-\beta+(1-\gamma)\frac{1-2\gamma}{4\alpha})-(\frac{1-2\gamma}{4\alpha}))^2 \\\frac{1-2\gamma}{4\alpha}
    \end{pmatrix}\bigg\}$ & $x=\frac{1}{16\alpha^2\beta-1}+1$\vspace{1mm} \\ 
         $0<\beta<\frac{1}{16\alpha^2}$& $S=\bigg \{\begin{pmatrix}
        0\\0\\0
    \end{pmatrix},\begin{pmatrix}
       4\alpha\\1-4\gamma+2\gamma^2+(1-\gamma)\sqrt{K}\\1-2\gamma+\sqrt{K}
    \end{pmatrix}\bigg\}$, & $x=2\alpha(1+K^{-1})$\\
      & $K=(1-2\gamma)^2-16\alpha\beta$ &
         \\ \bottomrule 
    \end{tabular}
    
    \caption{Numerical counterexamples to the pointwise extensibility of \eqref{consist} when $\alpha>0$, $\Delta \neq 0,\frac{1}{16\alpha^2}$.}
    \label{tab:alpha}
\end{table}

 \com{These counterexamples are one-dimensional, but they can straightforwardly be extended to any dimension by considering vectors with one non-zero entry.}
   \end{proof}
   
\com{Hence, the function classes admitting a simple, in the sense of \eqref{generalConstraint}, tight discrete description are in fact limited to variations of smooth (weakly/strongly)-convex functions. Characterizing other interesting function classes via the constraint-based approach would thus require starting from larger classes of constraints.} 
\section{Conclusion}\label{sec:conclu}
We considered the questions of proposing candidate interpolation constraints and of assessing the interpolability of these candidates, and proposed tools to gain insight into both questions. First, we proposed a constraint-centered approach to interpolation, starting from meaningful interpolation constraints to the class they define when imposed everywhere. This allows highlighting function classes based on the fact that they possess a simple tight description rather than on their analytical properties. First-order methods can then be tightly analyzed/designed for these function classes. Second, we introduced pointwise extensibility, which allows verifying the interpolability of a constraint algebraically.

\com{Relying on these tools, we exhaustively analyzed a class of potential constraints and obtained interpolation constraints for a class of quadratic functions. Then, relying on pointwise extensibility, we provided interpolation constraints for weakly convex functions satisfying a continuity condition, and tightly numerically analyzed the performance of the subgradient method on this class. Finally, we provided a heuristic numerical approach to assess non-interpolability of constraints, and exploited it to show that several commonly used descriptions were, in fact, not interpolation constraints.} 

Regarding interpolation constraints, several questions remain open. For instance, it is not known how to combine interpolation constraints defining different properties to tightly define function classes combining those properties, as in the case of smooth convex functions. In general, there is no principled way to design interpolation constraints for a function class $\F$. Then, no interpolation constraint involving high-order quantities is presently available. Finally, it is not known what constraints to impose on a data set to ensure interpolation of the set by a function of interest defined on a given subset of $\R^d$ strictly different from $\R^d$, since classical interpolation constraints are not valid in this case. In future research, we would like to extend our framework to answer some of these questions.
\paragraph{Codes}\label{rem:code}
All numerical results presented in this paper, and in particular the heuristic presented in Section \ref{sec:counterex}, are available at 
\begin{center}
    \url{https://github.com/AnneRubbens/Func_Interp_Code/tree/main}.
\end{center}
We used the programming language Matlab \cite{MATLAB}, and the solver Mosek 10.1 \cite{mosek}.
\section*{Declarations}

\noindent \textbf{Funding}\\
A. Rubbens is supported by a FNRS fellowship and a L’Oréal-UNESCO For Women in Science fellowship. J. M. Hendrickx is supported by the “RevealFlight” Concerted Research Action (ARC) of the Federation Wallonie-Bruxelles, by the Incentive Grant for Scientific Research (MIS) “Learning from Pairwise Comparisons” of the F.R.S.-FNRS, and by the SIDDARTA Concerted Research Action (ARC) of the Federation Wallonie-Bruxelles.

\noindent\textbf{Conflict of interests} \\
The authors declare that they have no conflict of interest or competing interests.

\bibliographystyle{spmpsci}      
\bibliography{biblio.bib}   
\appendix
\section{Instance of a PEP program}\label{app:PEP}
The PEP framework, introduced in \cite{drori2014performance}, computes tight worst-case bounds on the performance of a method on a function class $\F$ by formulating the search of a worst-case function as an optimization problem over $f\in \mathcal{F}$. For instance, to analyze the gradient descent after $N$ iterations on smooth convex functions, when the initial condition is chosen to be $\|x_1-x_\star\|^2\leq R^2$ and the performance measure is chosen to be $f(x_N)-f(x_\star)$, where $x_\star$ is a minimizer of $f$, the idea is to solve
 \begin{equation}\label{PEP_untractable}
    \begin{aligned}
   &  && \underset{f \in \F_{0,L}}{\max \ } f(x_N)-f(x_\star)\\
     &\text{Definition of the method:}&&\text{ s.t. } x_{i+1}=x_i-\alpha \nabla f(x_i), \quad i\in [N-1] \\
     &\text{Initial condition:}&&\quad \quad\|x_0-x_\star\|^2\leq R^2 \\
     &\text{Optimality:}&&\quad \quad \nabla f(x_\star)=0.  
\end{aligned}
\end{equation}

Without loss of generality, $x_\star$ and $f(x_\star)$ can be set to $0$. This infinite-dimensional problem is made tractable by optimizing over the finite set $S=\{(x_i,f_i,g_i)\}_{i\in [N]\cup _\star}$ of iterates and optimum, under the constraint that $S$ is $\F$-interpolable. In the same spirit as in traditional performance analysis, the abstract concept $f\in \F$ is thus translated into a set of algebraic constraints imposed on $S$, and Problem \eqref{PEP_untractable} becomes
 \begin{equation}\label{PEP_tractable}
    \begin{aligned}
   &  && \underset{S=\{(x_i,f_i,g_i)\}_{i\in [N]\cup _\star}}{\max \ } f_N-f_\star\\
     &\text{Definition of the method:}&&\text{ s.t. } x_{i+1}=x_i-\alpha g_i, \quad i\in [N-1] \\
     &S\text{ is } \mathcal{F}_{0,L}\text{-interpolable: }&&\quad \quad f_i\geq f_j+\langle g_j,x_i-x_j\rangle +\frac{1}{2L}\|g_i-g_j\|^2, \ i,j\in [N]\cup _\star\\
     &\text{Initial condition:}&&\quad \quad\|x_1-x_\star\|^2\leq R^2 \\
     &\text{Optimality:}&&\quad \quad g_\star=0.  
\end{aligned}
\end{equation}

Whenever interpolation constraints for $\F$ are known and imposed, the resulting PEP-based bound is tight and the set $S$ of iterates corresponds to an actual worst-case instance. In fact, the dual variables associated to each constraint are the coefficients associated to the optimal combination of these constraints, in general challenging to obtain by hand. Otherwise, when only imposing necessary conditions on $S$, Problem \eqref{PEP_tractable} is a relaxation of Problem \eqref{PEP_untractable}, and the bound achieved could be conservative since $S$ might not be consistent with some  $f \in \mathcal{F}$. For instance, \com{letting $N=2$, solving Problem \eqref{PEP_tractable} and thus imposing the interpolation constraint \eqref{truesmoothconvexity} allows concluding $f(2)-f_\star\leq \frac{1}{6}$, while solving a relaxation of \eqref{PEP_tractable} by imposing \eqref{eq:smoothconvexity} on $S$ leads to a bound of $\frac{1}{4}$}. This implies that the solution obtained in this setting belongs to the admission region allowed by \eqref{eq:smoothconvexity} and not by \eqref{truesmoothconvexity}. In addition, \eqref{PEP_tractable} can be turned into a semi-definite program whenever all quantities involved are linear in function values and scalar products of points and subgradients, see, e.g., \cite{taylor2017smooth} for details.\vspace{-0.5cm}
\section{Comparison of several definitions of interpolation constraints}\label{app:equiv}
Classically, $p$ is defined as an interpolation constraint for a given function class $\F$ if any finite set $S=\{(x_i,f_i,g_i)\}_{i\in [N]}$ is $\F$-interpolable if and only if $S$ satisfies $p$ \cite{taylor2017smooth}.
It holds that if $p$ satisfies this classical definition, then $p$ is an interpolation constraint for the class it defines, in the sense of Definition \ref{def:intrp}.
 \begin{proposition}
    Let $p$ be a pairwise constraint. If $p$ serves as an interpolation constraint for a function class $\F$, in the sense of \cite{taylor2017smooth}, i.e., any finite set $S=\{(x_i,f_i,g_i)\}_{i\in [N]}$ is $\F$-interpolable if and only if $S$ satisfies $p$, then $p$ is an interpolation constraint in the sense of Definition \ref{def:intrp}.
\end{proposition}
\begin{proof}
\com{Let $p$ be an interpolation constraint for a function class $\F$. Consider any finite set $\X=\{x_i\}_{i\in[N]}\subset\R^d$ and any $F\in \F_p(\X)$.}

\com{First, we show $\F \subset \bF_p(\R^d)$. Consider any $F=(f,g)\in \F$, where $g(x)=\partial f(x)$, and suppose that, at some pair $x,y\in \R^d$, $F$ does not satisfy $p$. Then, by interpolability of $p$, there exists no function in $\F$ interpolating this pair, contradicting $F\in \F$. Hence, $F\in \F_p(\R^d)$. In addition, $F$ is maximal, since $p$ is differentially consistent and $g(x)=\partial f(x)$, hence $\F\subseteq \bF_p(\R^d)$. In addition, consider any finite $\X \subset \R^d$, and any $F \in \F_p(\X)$. We need to show existence of some $\bar F \in \bF_p(\R^d)$ such that $F(x)=\bar F(x), \ \forall x\in \X$. Since $S=\{(x_i,F_i)\}_{i\in [N]}$ satisfies $p$, it is interpolable by some $\bar F\in \F\subseteq \bF_p(\R^d)$, which concludes the proof.}
\end{proof}
The other direction does not hold: consider for instance the trivial constraint $p=0$, both an interpolation constraint as in Definition \ref{def:intrp} for the class $\F_p$ containing all mappings $F=(f,g)$, where $g(x)=\R^d$, $\forall x\in \R^d$, and an interpolation constraint in the classical sense for $\F$ the class of continuous functions.
\section{Detailed resolution of \eqref{PEP_wc}}\label{app:weaklyconvprog}
First, observe that in \eqref{PEP_wc}, imposing $C_{ij}$ to be the projection of $x_j$ onto $\mathcal{B}_{\frac{g_i}{\mu}+x_i, \frac{B}{\mu}}$ $\forall i,j$, i.e.,   \begin{align*}
    C_{ij}=\frac{g_i}{\mu}+x_i+\frac{B}{\mu}\frac{x_j-x_i-\frac{g_i}{\mu}}{\|x_j-x_i-\frac{g_i}{\mu}\|} \text{ if $\|\mu(x_j-x_i)-g_j\|\geq B$}, \quad C_{ij}=x_j \text{ otherwise}, 
\end{align*} proves difficult to implement, since neither linear nor quadratic in the variables of the problem. However, we show $\tilde p_{\FuB}^{ij}$ can be equivalently  expressed in a tractable way.
 \begin{lemma}
    Let $\tilde p_{\FuB}$ be defined as in \eqref{22}. Then, for any $x_i,x_j \in \R^d$, 
      \begin{align}
\tilde{p}_{\FuB}^{i,j}\geq 0\Leftrightarrow
\begin{cases}
   & \exists C_{ij}\in \R^d: \\
   &f_j \geq f_i+\langle g_i, x_j-x_i\rangle-\frac{\mu}{2}\|x_i-x_j\|^2+\frac{\mu}{2}\|x_j-C_{ij}\|^2\\
   & \|g_i+\mu(x_i- C_{ij})\|^2\leq B^2\\
   &\|g_i\|^2\leq B^2.
    \end{cases} \label{21}
\end{align}
\end{lemma}
\begin{proof}
    Denote by $C_{ij}^*$ the projection of $x_j$ onto $\mathcal{B}_{\frac{g_i}{\mu}+x_i, \frac{B}{\mu}}$. We prove that 
      \begin{align}
        &\quad\  \exists C_{ij}\in \mathcal{B}_{\frac{g_i}{\mu}+x_i, \frac{B}{\mu}}: \ f_j\geq f_i+\langle g_i, x_j-x_i\rangle-\frac{\mu}{2}\|x_i-x_j\|^2+\frac{\mu}{2}\|x_j-C_{ij}\|^2\label{1}\\
        &\Leftrightarrow 
        \ f_j\geq f_i+\langle g_i, x_j-x_i\rangle-\frac{\mu}{2}\|x_i-x_j\|^2+\frac{\mu}{2}\|x_j-C_{ij}^*\|^2. \label{2}
    \end{align}
    Clearly, \eqref{2} implies \eqref{1}. Suppose now \eqref{1} is satisfied. Since $\|x_j-C_{ij}\|\geq \|x_j-C_{x_iy}^*\|$, \eqref{2} is satisfied.
\end{proof}
We are now ready to propose a tractable program equivalent to Problem \eqref{PEP_wc}. Let
  \begin{align*}
    F&=[f_0,...,f_N, f_0^y, ..., f_N^y], \\P&=[x_0,g_0,...,g_N, g_0^y, ..., g_N^y, C_{x_0, x_1}, ..., C_{x_0,y_N}, C_{x_0,x^*}, ...,C_{y_0,x_0}, ..., C_{y_N,x^*}],
\end{align*}
where $C_{x_0, x_1}, ..., C_{y_N,x^*}$ englobe the variables $C$ associated to any pair of points in $\{x_i\}_{i\in [N]\cup^*}\cup \{y_i\}_{i\in [N]}$. Let $G=P^TP$. \com{We seek to obtain matrices and vectors $A_{ij}$, $c_{ij}$, $b$, and $A_R$ allowing to express \eqref{PEP_wc} as a semidefinite program.}

\com{We first define vectors allowing to access to the variables $x_i,\ y_i,\ g_i,\ g^y_i$ and $C_{ij}$ from $P$.} Let $e_i$ be the unit vector of index $i$. For $i\in [N]$, let 
  \begin{align*}
   &u_i= e_{i+2}, \ v_i=e_{N+i+2}, \ a_i=e_1-\sum_{k=1}^i h u_{k+1}\text{ and } b_i=e_1-\sum_{k=0}^{i} h v_k. 
\end{align*}
In addition, let $a^*=u^*=[0,...,0]$. Then, $g_i=Pu_i$, $g^y_i=Pv_i$, $x_i=Pa_i$, $y_i=Pb_i$, $x^*=Pa^*$ and $g^*=Pu^*$. For $i,j\in [N]$, $i\neq j$, let    \begin{align*}
    &c^{xx}_{ij}=e_{2(N+1)+2+iN+j}, &&\ c^{xy}_{ij}=e_{2(N+1)+2+N^2+iN+j}, \\& c^{yx}_{ij}=e_{2(N+1)+2+2N^2+iN+j}, &&\ c^{yy}_{ij}=e_{2(N+1)+2+3N^2+iN+j}\\&
    c^{x*}_i=e_{2(N+1)+2+4N^2+i}, &&\ c^{*x}_i=e_{3(N+1)+2+4N^2+i}, \\&  c^{y*}_i=e_{4(N+1)+2+4N^2+i},&&\ c^{*y}_i=e_{5(N+1)+2+4N^2+i}, 
\end{align*} implying $C_{x_i,x_j}=Pc^{xx}_{ij}$, $C_{x_i,y_j}=Pc^{xy}_{ij}$, $C_{y_i,x_j}=Pc^{yx}_{ij}$,  $C_{y_i,y_j}=Pc^{yy}_{ij}$, $C_{x_i,x^*}=Pc^{x*}_{i}$, $C_{y_i,x^*}=Pc^{y*}_{i}$, $C_{x^*,x_i}=Pc^{*x}_{i}$ and $C_{x^*,y_i}=Pc^{*y}_{i}$. 

\com{We now define matrices $A_R$ (initial condition), $A_{ij}$ ($f_j\geq f_i+\langle g_i,x_j-x_i\rangle-\frac{\mu}{2}\|x_i-x_j\|^2+\frac{\mu}{2}\|x_j-C_{+j}\|^2$), $E_{ij}$ ($\|g_i+\mu(x_i-C_{ij})\|^2\leq B^2$), $H_i$ ($\|g_i\|^2\leq B^2$) and  vector $\tau$ (objective) translating \eqref{PEP_wc} into a semidefinite program.}
    \begin{align*}
     A_R&=(a_0-b_0)(a_0-b_0)'
     \\ 2A^{xy}_{ij}&=(v_j(a_i-b_j)'+(a_i-b_j)v_j')-\mu(a_i-b_j)(a_i-b_j)'+\mu (a_i-c_{ij}^{xy})(a_i-c_{ij}^{xy})' \\ E^{xy}_{ij}&=(v_j+\mu (b_j-c_{ij}^{xy}))(v_j+\mu (b_j-c_{ij}^{xy}))'
     \\  H_i^x&=v_iv_i', \quad  H_i^y=u_i u_i' \\
      \tau&=\frac{1}{(N+1)\rho^2} \sum_{i=0}^N (a_i-b_i)(a_i-b_i)^2, 
 \end{align*}
 where $A^{xx}$, $A^{yx}$, $A^{yy}$, $E^{xx}$, $E^{yx}$, $E^{yy}$ are defined similarly as $A^{xy}$, $E^{xy}$ with the corresponding $a$, $b$, $v$ or $u$, to be able to impose interpolation conditions on any pair $(x_i,x_j)$, $(x_i,y_j)$, $(y_i,x_j)$ and $(y_i,y_j)$. Problem \eqref{PEP_wc} can then be rewritten as    \begin{equation}\label{PEP_wc_SDP}
    \begin{aligned}
   & \textcolor{white}{hihi}&&\underset{G\geq 0, F }{\max\ } Tr(\tau G)\\
     &\quad \quad \quad \  \text{ s.t. }&&\textcolor{white}{hihi}  \\
     &\tilde p_{\FuB}\geq 0:&&F(e_i-e_j)+Tr (GA_{ij}^{xx})\leq 0, &Tr(GE^{xx}_{ij})\leq B, \ i,j\in [N],* \\&\textcolor{white}{hihi} &&F(e_i-e_{j+N+1})+Tr (GA_{ij}^{xy})\leq 0, & Tr(GE^{xy}_{ij})\leq B, \ i,j\in [N],*
     \\&\textcolor{white}{hihi} && F(e_{i+N+1}-e_j)+Tr (GA_{ij}^{yx})\leq 0, & Tr(GE^{yx}_{ij})\leq B, \ i,j\in [N],*
     \\&\textcolor{white}{hihi} && F(e_{i+N+1}-e_{j+N+1})+Tr (GA_{ij}^{yy})\leq 0, &Tr(GE^{yy}_{ij})\leq B, \ i,j\in [N],*\\
     &\textcolor{white}{hihi}&& Tr(GH_i^x)\leq B, & Tr(GH_i^y)\leq B, i\in [N] \quad \ \quad 
     \\&\text{Initial cond.:}&& Tr (GA_r)\leq R^2 \\
     & \text{Optimality:} && Fe_i\geq0,\  Fe_{i+N+1}\geq f^*,\  i=0...,N.
     \\& G=PP^T:&& \text{rank } G\leq d
\end{aligned}
\end{equation}
Relaxing the rank constraint gives a convex semidefinite program and furnishes valid solutions for any $d\geq \text{ dim} G$ \cite{taylor2017smooth}. 
\section{Proof of Proposition \ref{propchiante}}\label{app:propchiante}
\begin{proof}
    First, note that given $p$ of the form \eqref{consist}, $p^{ij}\geq 0$ implies
     \begin{align*}
        f_i&\geq f_j+\langle g_j,x_i-_j\rangle+\alpha \|g_i-g_j+\frac{\gamma}{2\alpha}(x_i-x_j)\|^2+(\beta-\frac{\gamma^2}{4\alpha})\|x_i-x_j\|^2\\
        &\geq f_j+\langle g_j,x_i-x_j\rangle+(\beta-\frac{\gamma^2}{4\alpha}) \|x_i-x_j\|^2=f(y)+\langle g_j,x_i-x_j\rangle+\omega(\|x_i-x_j\|).
    \end{align*}
    Hence, $p^{ij}\geq 0$ implies \eqref{eq:subdiff_gx} and $p$ is differentially consistent.

    \ARj{Consider now a constraint $p$ of the form \eqref{generalConstraint} but that cannot be written as in \eqref{consist}. Consider $f:\R^d\to\R$ and $g:\R^d\to \R^d$ satisfying $p$ everywhere, and suppose that, at some $y\in \R^d$, $f$ is twice differentiable. For any $\varepsilon \in \R^d$, let $x=y+\varepsilon$, and consider the Taylor development of order 1 of $f$ and $g$ at $x$:
     \begin{align*}
        &f(x)=f_y+\langle \varepsilon,f'_y\rangle +\mathcal{O}(\varepsilon^2), \quad
        g(x)=g_y+ g'_y\varepsilon+\mathcal{O}(\varepsilon^2),
    \end{align*}
    where $f_y$, $g_y$, $f'_y$ and $g'_y$ denote, respectively, $f(y)$, $g(y)$, $\nabla f(y)$ and $\nabla^2f(y)$. Satisfaction of $p^{xy}$ then implies:
     \begin{align*}
        &f_y+\langle \varepsilon, f'_y\rangle +\mathcal{O}(\varepsilon^2)\geq f_y+B(\|g_y\|^2+2\langle g_y,g'_y\varepsilon\rangle)-(B+C)(\|g_y\|^2+\langle g_y,g'_y\varepsilon\rangle)\\&\quad\quad\quad\quad\quad\quad\quad\quad\quad+E(\|y\|^2+2\langle y,\varepsilon \rangle)+F\|y\|^2-(E+F)(\|y\|^2+\langle y,\varepsilon\rangle)\\&\quad\quad\quad\quad\quad\quad\quad\quad\quad+H(\langle g_y,y\rangle +\langle g_y,\varepsilon\rangle +\langle g'_y\varepsilon,y\rangle)+I(\langle g_y,y\rangle +\langle g'_y\varepsilon,y\rangle)\\&\quad\quad\quad\quad\quad\quad\quad\quad\quad+J(\langle g_y,y\rangle +\langle g_y,\varepsilon\rangle)-(I+J+K)\langle g_y,y\rangle +\mathcal{O}(\varepsilon^2)\\
        \Leftrightarrow& \langle \varepsilon, f_y'\rangle \geq (B-C)\langle g_y,g'_y\varepsilon\rangle)+(E-F)\langle y,\varepsilon \rangle+(H+J)\langle g_y,\varepsilon\rangle+(H+I) \langle g'_y\varepsilon,y\rangle) +\mathcal{O}(\varepsilon^2).
    \end{align*}   
    Since this holds for any choice of $\varepsilon$, it implies for any $k=1,\cdots,d$:
                 \begin{align*}
        f_y^{'(k)}=(H+J) g_y^{(k)}+(B-C)(g'_yg_y)^{(k)}+(E-F)y^{(k)}+(H+I)(g'_y y)^{(k)}.
    \end{align*}
Unless $B=C$, $E=F$, $H=-I$, and $H+J=1$, i.e., the parameters corresponding to \eqref{consist}, $g_y\neq f_y'$.}
\end{proof}
\end{document}